\documentclass[11pt]{amsart}
\usepackage{amsmath,amsfonts,amssymb,amsthm}
\usepackage{latexsym,bm,graphicx}
\usepackage{mathrsfs}
\usepackage{color}
\usepackage{hyperref}

\title[Noncompact $\textmd{RCD}(0,N)$ spaces with linear volume growth]{Noncompact $\textmd{RCD}(0,N)$ spaces with linear volume growth}
\author{Xian-Tao Huang}
\address{Yau Mathematical Sciences Center, Tsinghua University, Beijing, P. R. China, 100084, E-mail address: xthuang@math.tsinghua.edu.cn}

\newtheorem{thm}{Theorem}[section]
\newtheorem{prop}[thm]{Proposition}
\newtheorem{lem}[thm]{Lemma}

\newtheorem{cor}[thm]{Corollary}

\newtheorem{defn}[thm]{Definition}
\newtheorem{rem}[thm]{Remark}

\numberwithin{equation}{section}

\begin{document}

\maketitle
\begin{abstract} Since non-compact $\textmd{RCD}(0, N)$ spaces have at least linear volume growth, we study noncompact $\textmd{RCD}(0,N)$ spaces with linear volume growth in this paper.
One of the main results is that the diameter of level sets of a Busemann function grow at most linearly on a noncompact $\textmd{RCD}(0,N)$ space satisfying the linear volume growth condition.
Another main result in this paper is a splitting theorem at the noncompact end for a $\textmd{RCD}(0,N)$ space with strongly minimal volume growth.
These results generalize some theorems on noncompact manifolds with nonnegative Ricci curvature to non-smooth settings.


\vspace*{5pt}

\end{abstract}
\section{Introduction}  

Cheeger and Gromoll's splitting theorem (see \cite{CG71}) is one of the classical theorems on manifolds with nonnegative Ricci curvature.
In their proof, the Busemann function plays an important role.
There are some papers studying the Busemann function on manifolds, see \cite{Sor99} \cite{Sor98} etc.
Note that the definition of Busemann functions only concerns with the metric notion.
Combining with the notion of curvature bound (especially the Ricci curvature lower bound), we obtain many interesting properties and applications.

Recently, people are more and more interested in the study of non-smooth objects.
In the framework of metric measure spaces, there are lots of researches of Ricci curvature lower bounds.

A notion of `Ricci bounded from below by $K\in \mathbb{R}$ and dimension bounded above by $N\in[1,\infty]$' for general metric measure spaces is the so-called $\textmd{CD}(K,N)$-condition, which is introduced independently by Lott and Villani (\cite{LV09}) and by Sturm (\cite{St06I} \cite{St06II}).
For $N\in[1,\infty)$, the $\textmd{CD}(K,N)$ condition is compatibe with the Riemannian case and the space of $\textmd{CD}(K,N)$ spaces is stable under the  measured-Gromov-Hausdorff convergence.
In particular, they include Ricci-limit spaces (see \cite{CC97} \cite{CC00I} \cite{CC00II}).
Note that the metric measure space $(\mathbb{R}^{N},d_{\parallel\cdot\parallel},\mathcal{L}^{N})$, where $d_{\parallel\cdot\parallel}$ is the distance induced by some norm $\parallel\cdot\parallel$ and $\mathcal{L}^{N}$ is the $N$-dimensional Lebesgue measure, always satisfies $\textmd{CD}(0,N)$ (see p.926 in \cite{Vi09} for a proof).
Hence it is impossible to show a splitting theorem for a general $\textmd{CD}(0,N)$ space.
We remark that there is a so-called reduced curvature-dimension condition, denoted by $\textmd{CD}^{\ast}(K, N)$, introduced by Bacher and Sturm (\cite{BS10}).
For the case $K = 0$, $\textmd{CD}(0, N)$ and $\textmd{CD}^{\ast}(0, N)$ conditions are equivalent.
There is another version of Ricci curvature lower bound, called $\textmd{MCP}(K,N)$ condition, see Sturm \cite{St06II}, Ohta \cite{Oh07}. 

Recently, Ambrosio, Gigli and Savar\'{e} (\cite{AGS14}) introduced the notion of $\textmd{RCD}(K,\infty)$ spaces (see also \cite{AGMR15} for the simplified axiomatization), which rules out Finsler geometries.
The $\textmd{RCD}(K,N)$ (or $\textmd{RCD}^{*}(K,N)$) are considered by many authors, see \cite{AMS14} \cite{AMS15} \cite{EKS15} \cite{Gig15} etc.
Indeed, an $\textmd{RCD}(K,N)$ spaces (resp. $\textmd{RCD}^{*}(K,N)$) is a metric measure space which is both a $\textmd{CD}(K,N)$ (resp. $\textmd{CD}^{*}(K,N)$) and an $\textmd{RCD}(K,\infty)$ space.
Ricci-limit spaces are examples of $\textmd{RCD}(K,N)$ spaces, splitting theorem holds for Ricci-limit spaces by Cheeger and Colding \cite{CC96}.
Alexandrov spaces are also examples of $\textmd{RCD}(K,N)$ spaces.
An isometric splitting for Alexandrov spaces with non-negative Ricci curvature is established by Zhang and Zhu in \cite{ZZ10}, where `non-negative Ricci curvature' for Alexandrov spaces is defined by the authors in the same paper.
The splitting theorem on general $\textmd{RCD}(0,N)$ spaces is proved by Gigli:

\begin{thm}[Gigli, \cite{Gig13} \cite{Gig14}]\label{splitting}
Let $(X,d,m)$ be an $\textmd{RCD}(0,N)$ space containing a line.
Then $(X,d,m)$ is isomorphic to the product of the Euclidean line $(\mathbb{R},d_{Eucl},\mathcal{L}^{1})$ and another space $(X', d', m')$.
Moreover,
\begin{enumerate}
  \item if $N\geq 2$, then $(X', d', m')$ is a $\textmd{RCD}(0,N-1)$ space;
  \item if $N\in[1, 2)$, then $X'$ is just a point.
\end{enumerate}
\end{thm}
We note that the Busemann function also plays an important role in the proof of Theorem \ref{splitting}.

There are many other researches on $\textmd{RCD}(K,N)$ spaces, see \cite{GRS13}, \cite{MN14}, \cite{JZ16}, \cite{ZZ16}, \cite{GM14}, \cite{Ke15}, \cite{GMR14}, \cite{GigPhi15}, and so on.

Since $\textmd{RCD}(K,N)$ space satisfy Bishop-Gromov volume comparison estimate, it is not hard to prove that a noncompact $\textmd{RCD}(0, N)$ space has at least linear volume growth, see Proposition \ref{infinity-volume}.
In fact, the proof of Proposition \ref{infinity-volume} only make use of Bishop-Gromov volume comparison estimate, thus similar property also holds the weaker notions of metric measure spaces with `nonnegative Ricci curvature' such as $\textmd{MCP}(0,N)$, $\textmd{CD}(0,N)$.
This generalizes the famous theorem of Calabi \cite{Cala75} and Yau \cite{Y76} on noncompact manifolds with nonnegative Ricci curvature.
It is interesting to study noncompact metric measure spaces with `nonnegative Ricci curvature' and linear volume growth, which is the aim of this paper.
It turns out that the Busemann function associated to some geodesic ray also plays an important role in the study.

The following theorem is one of our main results:

\begin{thm}\label{thm1.3}
Suppose $(X,d,m)$ is a noncompact $\textmd{RCD}(0,N)$ space satisfying the minimal volume growth condition (\ref{min-vol}), and $b$ is the Busemann function associated to a geodesic ray $\gamma$.
Then the diameter of the level sets $b^{-1}(r)$ grow at most linearly.
More precisely, we have
\begin{align}\label{1.1}
\limsup_{R\rightarrow+\infty}\frac{\textmd{diam}(b^{-1}(R))}{R}\leq C_{0}\leq2,
\end{align}
where the diameter of $b^{-1}(R)$ is computed with respect to the distance $d$.
In particular, $b^{-1}(r)$ is compact.

\end{thm}

One important tool in this paper is the disintegration formula.
Any $1$-Lipschitz function $\varphi$ will give rise to an equivalence relation on the so-called transport set $\mathcal{T}$, such that each equivalence class is a geodesic ray.
Then one can establish a disintegration formula with respect to this equivalence relation.
Bianchini and Cavalletti have obtained more general results in \cite{BC13}, and use them as a tool to handle Monge problem in metric spaces, especially there are results on non-branching $\textmd{MCP}(K,N)$ spaces.
It is known that the $\textmd{RCD}^{*}(K,N)$ spaces satisfies the so-called essentially non-branching property, see \cite{RaSt14}.
In \cite{Ca14-2}, Cavalletti extents results in \cite{BC13} to handle the Monge problem in $\textmd{RCD}^{*}(K,N)$ spaces.
See also \cite{CaMo15-1} \cite{CaMo15-2} for the applications of this method in the proof of L\'{e}vy-Gromov isoperimetric inequality and some functional inequalities in metric measure space.
Busemann functions are very special $1$-Lipschitz functions, see Section \ref{sec3.1}.
We can apply the results in \cite{BC13} \cite{Ca14-2}.
Under the noncompact and $\textmd{RCD}(0,N)$ assumptions, we have a good monotonicity formula for the volume growth of level sets of the Busemann function, see Proposition \ref{prop3.19}.
Using the volume comparison properties we can prove Theorem \ref{thm1.3}.

The argument in the proof of Theorem \ref{thm1.3} can be easily modified to prove the following theorem:

\begin{thm}\label{thm1.1}
Suppose $(X,d,m)$ is a noncompact, non-branching $\textmd{MCP}(0,N)$ space satisfying the minimal volume growth condition (\ref{min-vol}), and $b$ is the Busemann function associated to a geodesic ray $\gamma$,
then the diameter of the level sets $b^{-1}(r)$ satisfies (\ref{1.1}).
In particular, $b^{-1}(r)$ is compact.

\end{thm}

Theorems \ref{thm1.3} and \ref{thm1.1} generalize Theorem 19 in Sormani's paper \cite{Sor98}.

Note that using the volume comparison method, Theorem 19 in \cite{Sor98} in fact gives the bound $C_{0}\leq6$.
So the proof of Theorem \ref{thm1.3} is a bit refinement from \cite{Sor98}.
We remark that in another paper \cite{Sor99}, Sormani proved that, for manifolds with nonnegative Ricci curvature and linear volume growth, the $C_{0}$ in (\ref{1.1}) can be chosen to be $0$, see Theorem 1 in \cite{Sor99}.
However, the proof of Theorem 1 in \cite{Sor99} makes use of Cheeger-Colding's almost rigidity theory, which is not available even for $\textmd{RCD}^{*}(K,N)$ spaces at present.
It is an interesting question whether we can generalize Theorem 1 in \cite{Sor99} to $\textmd{RCD}(0,N)$ spaces with linear volume growth.

Our next main result is also about noncompact $\textmd{RCD}(0,N)$ with minimal volume growth.
Following \cite{Sor99}, we introduce a notion of strongly minimal volume growth in Definition \ref{def5.1}, and then obtain a splitting theorem at the noncompact end for the $\textmd{RCD}(0,N)$ space with strongly minimal volume growth.

\begin{thm}\label{thm1.2}
Suppose $(X,d,m)$ is a noncompact $\textmd{RCD}(0,N)$ space and satisfies the strongly minimal volume growth (see Definition \ref{def5.1}),
then  $(X,d,m)$ has only one end and there is some metric measure space $(Z, d',\nu')$ such that one of the following holds:
\begin{enumerate}
  \item if $Z$ has exactly one point, then $(b^{-1}((r_{0},\infty)),d,m)$ is isomorphic to $((r_{0},\infty),d_{Ecul},c\mathcal{L}^{1})$ with $c=m(b^{-1}([r_{0},r_{0}+1]))$.
      Furthermore, in this case $(X, d)$ itself is isometric to some $([\bar{r},\infty),d_{Ecul})$.
  \item if $Z$ has more than one point, then $N\geq 2$, and $(Z,d',m')$ is a compact connected $\textmd{RCD}(0,N-1)$ space, and $(b^{-1}((r_{0},\infty)),d)$ is locally isometric to $(Z\times(r_{0},\infty), d'\times d_{\textmd{Ecul}})$.
      Furthermore, let $r_{1}=r_{0}+\frac{\textmd{diam}'(Z)}{2}$, then $(b^{-1}((r_{1},\infty)),d,m)$ is isomorphic to $(Z\times(r_{1},\infty), d'\times d_{\textmd{Ecul}}, \nu'\otimes\mathcal{L}^{1})$.
      Here $\textmd{diam}'(Z)$ is the diameter of $Z$ with respect to the distance $d'$.
\end{enumerate}
\end{thm}

Theorem \ref{thm1.2} generalizes Corollary 10 in \cite{Sor99} to non-smooth setting.

We remark that the conclusion that $(b^{-1}((r_{0},\infty)),d)$ and $(Z\times(r_{0},\infty), d'\times d_{Ecul})$ are local isometric cannot be replaced by isometric, see the following example.

Let $S^{2}\hookrightarrow \mathbb{R}^{3}$ be the unit sphere equipped with the round metric $g_{round}$, and $\tau:S^{2}\rightarrow S^{2}$ be the standard inversion given by $\tau(x_{1},x_{2},x_{3})=(-x_{1},-x_{2},-x_{3})$, which is an isometry of $(S^{2},g_{round})$.
Suppose $(S^{2}\times \mathbb{R},g_{cyl}=g_{round}+dr^{2})$ is the standard metric on the cylinder,
then $\sigma:(S^{2}\times \mathbb{R},g_{cyl})\rightarrow(S^{2}\times \mathbb{R},g_{cyl})$, $\sigma((x,r))=(\tau(x),-r)$ is obviously an isometry.
Denote the quotient space with respect to $\sigma$ by $(X,g_{X})$, and we denote $[(x,r)]\in X$ the quotient point of $(x,r)$ and $\sigma((x,r))$.
It is easy to see that $b:X\rightarrow\mathbb{R}^{+}$, $b([(x,r)]):=|r|$ is a well-defined Busemann function.
Let $r_{0}=\frac{1}{10}$.
Note that $X$ inherits a metric measure structure $(X,d_{X},m_{X})$ from $(X,g_{X})$, and $b^{-1}((r_{0},\infty))$ inherit the metric measure structure $(b^{-1}((r_{0},\infty)),d_{X},m_{X})$ as a subspace of $X$.
Similarly, $S^{2}\times (r_{0},\infty)$ inherits a metric measure structure $(S^{2}\times (r_{0},\infty),d_{cyl},m_{cyl})$ from $(S^{2}\times (r_{0},\infty),g_{cyl})$.
Obviously, $(b^{-1}((r_{0},\infty)),g_{X})$ is isometric to the half cylinder $(S^{2}\times (r_{0},\infty),g_{cyl})$ as Riemannian manifolds,
but this only induces a local isometry between $(b^{-1}((r_{0},\infty)),d_{X})$ and $(S^{2}\times (r_{0},\infty),d_{cyl})$.
In fact we have $d_{X}([(x,\frac{1}{5})],[(\tau(x),\frac{1}{5})])=\frac{2}{5}<\pi=d_{cyl}((x,\frac{1}{5}),(\tau(x),\frac{1}{5}))$ for any $x\in S^{2}$.

We remark that recently Gigli and De Philippis have proved `volume cone implies metric cone' in the setting of $\textmd{RCD}(K,N)$ spaces in a preprint \cite{GigPhi15}.
The proof is very involved and lengthy, and it relies on the tools and results recently developed in \cite{Gig15}, \cite{Gig14-2}, \cite{Gig13}, \cite{GigHan15} etc.
As is indicated in \cite{GigPhi15}, the techniques in \cite{GigPhi15} can be adapted to other setting.
In Section \ref{sec5}, we will prove that the strongly minimal volume growth condition implies that there is a measure preserving map from $(b^{-1}((r_{0},\infty)),m)$ to some $(X'\times\mathbb{R}^{+},m'\otimes\mathcal{L}^{1})$ (see Proposition \ref{prop5.3}).
Hence Theorem \ref{thm1.2} is essentially a `volume cone implies metric cone'-type theorem in the setting of noncompact $\textmd{RCD}(0,N)$ spaces.
Most of our remaining arguments in the proof of Theorem \ref{thm1.2} follow the strategy in \cite{GigPhi15} closely.
Since \cite{GigPhi15} has provided a complete and clear proof for the `volume cone implies metric cone'-type theorem, in this paper we just highlight some detailed calculations due to different backgrounds and also concentrate on the geometric outcome for the case we are dealing with, the readers can refer to \cite{GigPhi15} for more details of the proof.
Since Theorem \ref{thm1.2} is a kind of local version of splitting theorem on $\textmd{RCD}(0,N)$ spaces, some of the calculations here share the similarity to those in \cite{Gig13} \cite{Gig14}.


\vspace*{10pt}

\noindent\textbf{Acknowledgments.} The author would like to thank Binglong Chen, Huichun Zhang and Xiping Zhu for their encouragements and helpful discussions.

\section{Preliminaries}

\subsection{Metric measure space}
Throughout this paper, we will always assume the metric measure space $(X,d,m)$ we consider satisfies the following: $(X, d)$ is a complete separable geodesic space, and $m$ is a nonnegative Radon measure with respect to $d$ and finite on bounded sets, $\textmd{supp}(m)=X$.

A curve $\gamma:[0,T]\rightarrow X$ is called a geodesic provided $d(\gamma_{s},\gamma_{t})=L(\gamma|_{[s,t]})$ for every $[s,t]\subset[0,T]$, where $L(\gamma)$ means the length of the curve $\gamma$.
$(X,d)$ is called a geodesic space if every two points $x,y\in X$ are connected by a geodesic $\gamma$.

$(X,d)$ is called non-branching if for any two geodesics $\gamma^{1},\gamma^{2}:[0,1]\rightarrow X$ satisfy  $\gamma^{1}|_{I}=\gamma^{2}|_{I}$ for some interval $I\subset[0,1]$, then $\gamma^{1}\equiv\gamma^{2}$ on $[0,1]$.

Let $\textmd{Geo}(X)\subset \textmd{Lip}([0,1],X)$ be the set of all geodesics with domain $[0,1]$.
We equip $\textmd{Lip}([0,1],X)$ with the uniform topology.
For $t\in[0,1]$, define the evaluation map $e_{t}:\textmd{Geo}(X)\rightarrow X$ by $e_{t}(\gamma)=\gamma(t)$.
Obviously, $e_{t}(\gamma)$ is continuous.

A map $\gamma:[0,\infty)\rightarrow X$ is called a geodesic ray if for any $T>0$ its restriction to $[0,T]$ is a geodesic.
A map $\gamma:\mathbb{R}\rightarrow X$ is called a line if for any $s,t\in \mathbb{R}$. $d(\gamma_{s},\gamma_{t})=|s-t|$.
We will always assume that geodesic rays and lines are parametrized by unit speed.

Two metric measure spaces $(X_{1},d_{1},m_{1})$, $(X_{2},d_{2},m_{2})$ with $\textmd{supp}(m_{1})=X_{1}$, $\textmd{supp}(m_{2})=X_{2}$, are said to be isomorphic provided there exists an isometry $T:(X_{1},d_{1})\rightarrow (X_{2}, d_{2})$ such that $T_{*}m_{1}=m_{2}$.

We denote by $\mathcal{B}(X)$ the space of all Borel sets in $X$.
Denote by $\mathcal{P}(X)$ the space of Borel probability measures on $X$, and $\mathcal{P}_{2}(X)\subset \mathcal{P}(X)$ the space of Borel probability measures $\xi$ satisfying $\int_{X} d^{2}(x,y)\xi(dy)<\infty$ for some (and hence all) $x\in X$.
Denote by $\mathcal{A}(X)$ the $\sigma$-algebra generated by all analytic subsets in $X$.


\subsection{Calculus on metric measure spaces}

Given a function $f\in C(X)$, the pointwise Lipschitz constant of $f$ at $x$ is defined as
$$\textmd{lip}(f)(x):=\limsup_{y\rightarrow x}\frac{|f(x)-f(y)|}{d(x, y)}\in[0,+\infty]$$
if $x$ is not isolated, and put $\textmd{lip}(f)(x)=0$ if $x$ is isolated.

\begin{defn}\label{test-plan}
Let $\pi\in \mathcal{P}(C([0, 1], X))$. We say that $\pi$ is a test plan if
\begin{enumerate}
  \item there exist a constant $C>0$ such that $(e_{t})_{*}\pi\leq Cm$, $\forall t\in[0,1]$,
  \item $\int\int_{0}^{1}|\dot{\gamma_{t}}|^{2}dt d\pi(\gamma)<\infty.$
\end{enumerate}
We adopt the convention that $\int_{0}^{1}|\dot{\gamma_{t}}|^{2}dt=+\infty$ provided $\gamma$ is not absolutely continuous, so any test plan must be concentrated on absolutely continuous curves.
\end{defn}

\begin{defn}[see \cite{AGS14II}]
The Sobolev class $S^{2}(X)$ is the space of all Borel functions $f:X\rightarrow \mathbb{R}$ for which there is a nonnegative function $G\in L^{2}(X)$ such that
$$\int|f(\gamma_{1})-f(\gamma_{0})|d\pi(\gamma)\leq\int\int_{0}^{1}G(\gamma_{t})|\dot{\gamma_{t}}|dt d\pi(\gamma)$$
for every test plan $\pi$.
Any such $G$ is called weak upper gradient for $f$.
For $f\in S^{2}(X)$ there exists a minimal $G$ in the $m$-a.e. sense, which is called minimal weak upper gradient and will be denoted by $|Df|$.

The Sobolev space $W^{1,2}(X)$ is defined as $L^{2}(X)\cap S^{2}(X)$ and is equipped with the norm
$$\parallel f\parallel^{2}_{W^{1,2}}:=\parallel f\parallel^{2}_{L^{2}}+\parallel |Df|\parallel^{2}_{L^{2}}.$$
\end{defn}

We define $S^{2}_{loc}(X)$ (resp. $W^{1,2}_{loc}(X))$ to be the space of functions locally equal to some function in $S^{2}(X)$ (resp. $W^{1,2}(X)$).
For $U\subset X$ open, we define $S^{2}(U)$ (resp. $W^{1,2}(U))$) to be the space of functions locally in $U$ equal to some function in $S^{2}(X)$ (resp. $W^{1,2}(X)$) such that $|Df|\in L^{2}(U)$  (resp. $f, |Df|\in L^{2}(U))$.
We can also define the spaces $S^{2}_{loc}(U)$ or $W^{1,2}_{loc}(U)$.

Note that every Lipschitz function $f$ belongs to $W^{1,2}_{loc}(X)$ and satisfies
$$|Df|\leq \textmd{lip}(f), \qquad m\text{-a.e.}.$$
The equality may not hold for general metric measure space.
However, by the results in \cite{C99} and \cite{AGS14II}, if $(X,d,m)$ support a local doubling property and a weak local $(1,1)$-Poincar\'{e} inequality, then for any Lipschitz function $f$, $|Df|= \textmd{lip}(f)$ holds $m$-a.e..

The Dirichlet energy $\mathcal{E}:L^{2}(X)\rightarrow[0,\infty]$ is defined to be
\begin{align}
\mathcal{E}(f):=\left\{
               \begin{array}{ll}
                 \frac{1}{2}\int|Df|^{2}dm, & \hbox{if $f\in W^{1,2}(X)$;} \\
                 +\infty, & \hbox{otherwise.}
               \end{array}
             \right.
\end{align}

\begin{defn}
$(X,d,m)$ is called an infinitesimally Hilbertian space if $W^{1,2}(X)$ is an Hilbert space.
\end{defn}

On an infinitesimally Hilbertian space $(X,d,m)$, for any $f,g\in S^{2}_{loc}(X)$, the map $\langle\nabla f,\nabla g\rangle:X\rightarrow\mathbb{R}$ is $m$-a.e. defined to be
$$\langle\nabla f,\nabla g\rangle:=\inf_{\epsilon>0}\frac{|D(g+\epsilon f)|^{2}-|Dg|^{2}}{2\epsilon},$$
where the infimum is in $m$-essential sense.
Obviously, $\langle\nabla f,\nabla f\rangle=|Df|^{2}$.
The infinitesimal Hilbertianity make the map
$S^{2}(X)\ni f, g\mapsto \langle\nabla f,\nabla g\rangle\in L^{1}(X)$
bilinear and symmetric.
Furthermore, $\langle\nabla f,\nabla g\rangle$ satisfies the chain rule and Leibniz rule (see \cite{Gig15}).

If $(X,d,m)$ is infinitesimally Hilbertian, given an open set $U\subset X$, $D(\mathbf{\Delta},U)\subset W^{1,2}_{loc}(U)$ is the space of Borel functions $f\in W^{1,2}_{loc}(U)$ such that there exists a signed Radon measures $\mu$ on $U$ such that
\begin{align}
\int g d\mu =-\int \langle\nabla f,\nabla g\rangle dm
\end{align}
holds for any $g: X\rightarrow\mathbb{R}$ Lipschitz with $\textmd{supp}(g)\subset\subset U$.
$\mu$ is uniquely characterized and we denote it by $\mathbf{\Delta}f|_{U}$.
In case $U=X$ we simply write $g\in D(\mathbf{\Delta})$ and $\mu=\mathbf{\Delta}g$.

The space $D(\Delta)\subset W^{1,2}(X)$ is the space of functions $f$ for which there is a function in $L^{2}(X)$, called the Laplacian of $f$ and denoted by $\Delta f$, such that
\begin{align}
\int g\Delta f dm =-\int \langle\nabla f,\nabla g\rangle dm,\qquad \forall g\in W^{1,2}(X).
\end{align}

If $(X,d,m)$ is a proper infinitesimally Hilbertian space, then by Proposition 4.24 in \cite{Gig15}, for any $f\in W^{1,2}(X)$, $f\in D(\Delta)$ if and only if $f\in D(\mathbf{\Delta})$ with $\mathbf{\Delta}f=hm$ for some $h\in L^{2}(X,m)$.
Furthermore, if this holds then we have $h =\Delta f$.

In \cite{Gig14-2}, the notions of tangent and cotangent modules of a metric measure space $(X,d,m)$ are introduced. Denote the tangent and cotangent modules by $L^{2}(TX)$ and $L^{2}(T^{*}X)$ respectively.
The pointwise norm on both spaces will be denoted by $|\cdot|$.
The differential of a function $f\in W^{1,2}(X)$ is an element $df\in L^{2}(T^{*}X)$ defined in Section 2.2.2 of \cite{Gig14-2}.
The differential operator $d$ satisfies the locality property, chain rule and Leibniz rule.
For $f\in W^{1,2}(X)$, $|df|=|Df|$ holds $m$-a.e..
In case $X$ is infinitesimally Hilbertian, the gradient $\nabla f\in L^{2}(TX)$ of $f\in W^{1,2}(X)$ is the unique element associated to the differential $df$ via the Riesz isomorphism for modules.

The space $D(\textmd{div})\subset L^{2}(TX)$ is the set of all vector fields $V$ for which there exists $f\in L^{2}(X)$ such that
$$\int fg dm = -\int dg(V) dm, \qquad \forall g \in W^{1,2}(X).$$
$f$ is uniquely characterized, we call it the divergence of $V$ and denote it by $\textmd{div}(V)$.
The Leibniz rule holds for the divergence.
Suppose $(X,d,m)$ is infinitesimally Hilbertian, $f\in W^{1,2}(X)$, then $\nabla f\in D(\textmd{div})$ iff $f\in D(\Delta)$, and in this holds, we have $\textmd{div}(\nabla f)=\Delta f$.
See Section 2.3.3 in \cite{Gig14-2} for more details about the divergence.

\subsection{Optimal transport}
Let $c:X\times X\rightarrow\mathbb{R}$ be the function $c(x,y)=\frac{d^{2}(x,y)}{2}$.
For $\mu,\nu\in{\mathcal{P}_{2}(X)}$, consider their Wasserstein distance $W_{2}(\mu,\nu)$ defined by
\begin{align}\label{2.4}
W_{2}^{2}(\mu,\nu)=\underset{\eta\in\Gamma(\mu,\nu)}{\min}\int_{X\times{X}}d^{2}(x,y)d\eta(x,y),
\end{align}
where $\Gamma(\mu,\nu)$ is the set of Borel probability measures $\eta$ on $X\times{X}$ satisfying $\eta(A\times{X})=\mu(A)$, $\eta(X\times{A})=\nu(A)$ for every Borel set $A\subset{X}$.
We call a plan $\eta$ that minimizes (\ref{2.4}) an optimal plan.

Since $(X, d)$ is a geodesic space, $W_{2}$ can be equivalently characterized as:
\begin{align}\label{2.3}
W_{2}^{2}(\mu,\nu)={\min}\int\int_{0}^{1}|\dot{\gamma}_{t}|^{2}dtd\pi(\gamma),
\end{align}
where the minimum is taken among all $\pi\in \mathcal{P}(C([0, 1],X))$ such that $(e_{0})_{*}\pi=\mu$ and $(e_{1})_{*}\pi=\nu$. 
The set of optimal geodesic plans realizing the minimum in (\ref{2.3}) is denoted by $\textmd{OptGeo}(\mu,\nu)$.

Given a map $\varphi : X\rightarrow \mathbb{R}\cup \{-\infty\}$, its $c$-transform $\varphi^{c} : X\rightarrow \mathbb{R}\cup \{-\infty\}$ is defined to be
$$\varphi^{c}(y)=\inf_{x\in X}\bigl[\frac{d^{2}(x,y)}{2}-\varphi(x)\bigr].$$

A function $\varphi : X\rightarrow \mathbb{R}\cup \{-\infty\}$ is called $c$-concave provided it is not identically $-\infty$ and it holds $\varphi=\psi^{c}$ for some $\psi$.

For a general function $\varphi : X\rightarrow \mathbb{R}\cup \{-\infty\}$ it always holds
\begin{align}\label{2.2}
\varphi^{cc}\geq \varphi.
\end{align}

It turn out that $\varphi$ is $c$-concave if and only if $\varphi^{cc}= \varphi$.
The $c$-superdifferential $\partial^{c}\varphi$ of the $c$-concave function $\varphi$ is defined by
$$\partial^{c}\varphi:=\{(x,y)\in X\times X| \varphi(x) + \varphi^{c}(y)=\frac{d^{2}(x, y)}{2}\}.$$
The set $\partial^{c}\varphi(x)\subset X$ is the set of those $y$'s such that $(x, y)\in \partial^{c}\varphi$.

A subset $\Upsilon\subset X\times X$ is said to be $c$-cyclically monotone if for any $N \in \mathbb{N}$ and  $\{(x_{i}, y_{i})\}_{i\leq N}\subset \Upsilon$, we have (with the convention $y_{N+1}=y_{1}$)
\begin{align}
\sum_{i=1}^{N}c(x_{i},y_{i})\leq \sum_{i=1}^{N}c(x_{i},y_{i+1}).
\end{align}

It is well known that optimal plan, the $c$-concave function and $c$-cyclically monotone set are closely related to each other, see Theorem 5.10 in \cite{Vi09} or Theorem 2.13 in \cite{AG11}.

\subsection{The curvature-dimension conditions}

A dynamical transference plan $\Xi$ is a Borel probability measure on $\textmd{Geo}(X)$, and the path
$\{\xi_{t}\}_{t\in[0,1]}\subset \mathcal{P}_{2}(X)$ given by $\xi_{t}=(e_{t})_{*}\Xi$ is called a displacement interpolation associated to $\Xi$.
For $K\in \mathbb{R}$, define the function $s_{K}:[0,+\infty)\rightarrow\mathbb{R}$ (on $[0,\pi/\sqrt{K})$ if $K> 0$) as
\begin{align}
s_{K}(t):=\left\{
\begin{array}{cc}
  (1/\sqrt{K})\sin(\sqrt{K}t) & \text{ if } K>0, \\
  t & \text{ if } K=0, \\
  (1/\sqrt{-K})\sinh(\sqrt{-K}t) & \text{ if } K<0.
\end{array}
\right.
\end{align}

\begin{defn}\label{def2.8}
We say a metric measure space $(X, d, m)$ satisfies the $(K, N)$-measure contraction property ($\textmd{MCP}(K, N)$) if for every point $x\in X$ and $m$-measurable set $A\subset X$ with $m(A)\in(0,\infty)$ there exists a displacement interpolation $\{\xi_{t}\}_{t\in[0,1]}$ associated to a dynamical transference plan $\Xi=\Xi_{x,A}$ satisfying the following:
\begin{enumerate}
  \item $\xi_{0}=\delta_{x}$ and $\xi_{1}=\frac{1}{m(A)}m|_{A}$;
  \item for every $t\in[0,1]$, we have
  \begin{align}\label{2.10}
  dm\geq(e_{t})_{*}\biggl( t\biggl\{\frac{s_{K}(td(x,\gamma(1))/\sqrt{N-1})}{s_{K}(d(x,\gamma(1))/\sqrt{N-1})}\biggr\}^{N-1} m(A)d \Xi(\gamma)\biggr).
  \end{align}
\end{enumerate}
\end{defn}

If $(X,d,m)$ is an $\textmd{MCP}(K, N)$ space for $N\in[1,\infty)$, then by \cite{Oh07}, the Bishop-Gromov volume comparison estimate holds.
Furthermore, we can derive that $(X,d)$ is locally compact.

For $N\in(1,\infty)$ we define the functional $\mathcal{U}_{N}:\mathcal{P}_{2}(X)\rightarrow[-\infty,0]$ to be
$$\mathcal{U}_{N}(\mu):=-\int\rho^{1-\frac{1}{N}}dm,\quad \mu=\rho m+\mu^{s}, \mu^{s}\bot m,$$
and for $N=1$, we define $\mathcal{U}_{1}:\mathcal{P}_{2}(X)\rightarrow[-\infty,0]$ to be
$$\mathcal{U}_{1}(\mu):=-m(\{\rho>0\}),\quad \mu=\rho m+\mu^{s}, \mu^{s}\bot m.$$

For $N \in[1,\infty)$, $K \in \mathbb{R}$ the distortion coefficients $\tau^{(t)}_{K,N}(\theta)$ are functions $[0,1]\times [0,\infty)\ni(t, \theta)\mapsto \tau^{(t)}_{K,N}(\theta)\in [0,+\infty]$ defined by
\begin{align}
\tau^{(t)}_{K,N}(\theta)
=\left\{
   \begin{array}{ll}
     +\infty, & \hbox{if $K\theta^{2}\geq (N-1)\pi^{2}$;} \\
     t^{\frac{1}{N}}\biggl( \frac{\sin(t\theta\sqrt{K/(N-1)})}{\sin(\theta\sqrt{K/(N-1)})}\biggr)^{1-\frac{1}{N}}, & \hbox{if $0<K\theta^{2}<(N-1)\pi^{2}$;} \\
     t, & \hbox{if $K\theta^{2}=0$ or} \\
        & \hbox{if $K\theta^{2}<0$ and $N=1$;} \\
     t^{\frac{1}{N}}\biggl( \frac{\sinh(t\theta\sqrt{-K/(N-1)})}{\sinh(\theta\sqrt{-K/(N-1)})}\biggr)^{1-\frac{1}{N}}, & \hbox{if $K\theta^{2}<0$ and $N>1$.}
   \end{array}
 \right.
\end{align}

\begin{defn}\label{def2.6}
Given two numbers $K,N\in\mathbb{R}$ with $N\geq1$ we say $(X,d,m)$ satisfies the curvature-dimension condition $CD(K,N)$ if and only if for any $\mu_{0},\mu_{1}\in\mathcal{P}_{2}(X)$ with bounded support and $\mu_{0}=\rho_{0}m$, $\mu_{1}=\rho_{1}m$, there exist an optimal coupling $\pi$ of $\mu_{0},\mu_{1}$ and a geodesic $\{\mu_{t}\}_{t\in[0,1]}\subset\mathcal{P}_{2}(X)$ connecting $\mu_{0},\mu_{1}$ such that
\begin{align}\label{2.1}
\mathcal{U}_{N'}(\mu_{t})\leq -\int_{X\times X}\bigl[&\tau^{(1-t)}_{K,N'}(d(x_{0},x_{1}))\rho_{0}^{-\frac{1}{N'}}(x_{0})\\
&+\tau^{(t)}_{K,N'}(d(x_{0},x_{1}))\rho_{1}^{-\frac{1}{N'}}(x_{1})\bigr]d\pi(x_{0},x_{1})\nonumber
\end{align}
for all $t\in[0,1]$ and $N'\in[N,\infty)$.
\end{defn}

Let $\sigma^{(t)}_{K,N-1}(\theta):=\bigl[t^{-\frac{1}{N}}\tau^{(t)}_{K,N}(\theta)\bigr]^{\frac{N}{N-1}}$.
\begin{defn}
We say $(X,d,m)$ satisfies the reduced curvature-dimension condition $CD^{*}(K,N)$ if we replace in Definition \ref{def2.6} the coefficients $\tau^{(t)}_{K,N}(\theta)$ by $\sigma^{(t)}_{K,N}(\theta)$.
\end{defn}


\begin{defn}
A $\textmd{CD}^{*}(K, N)$ space which is also infinitesimally Hilbertian is called an $\textmd{RCD}^{*}(K, N)$ space.
\end{defn}

Obviously, $\textmd{CD}(0, N)$ and $\textmd{CD}^{\ast}(0, N)$ conditions are equivalent.
In this paper, we use $\textmd{RCD}(0, N)$ instead of $\textmd{RCD}^{*}(0, N)$.

An $\textmd{RCD}^{*}(K,N)$ space also satisfy the $\textmd{MCP}(K, N)$ condition (see \cite{GRS13}).

The Bishop-Gromov volume comparison estimate holds on $\textmd{RCD}^{*}(K,N)$, in particular, for an $\textmd{RCD}(0, N)$ space, we have
\begin{align}\label{B-G}
\frac{m(B_{x}(r))}{m(B_{x}(R))}\geq\frac{r^{N}}{R^{N}}, \quad \forall x\in X, 0\leq r\leq R.
\end{align}
From (\ref{B-G}), we are easy to obtain
\begin{align}\label{B-G-2}
\frac{m(B_{x}(R))}{R^{N}}\geq\frac{m(B_{x}(R))-m(B_{x}(r))}{R^{N}-r^{N}}, \quad \forall x\in X, 0\leq r\leq R.
\end{align}
An interesting application of (\ref{B-G-2}) is Proposition \ref{infinity-volume}.
Proposition \ref{infinity-volume} generalizes the famous theorem on noncompact manifolds with nonnegative Ricci curvature proved independently by Calabi \cite{Cala75} and Yau \cite{Y76}.
\begin{prop}\label{infinity-volume}
Suppose $(X,d,m)$ is a noncompact $\textmd{RCD}(0,N)$ space with $N\geq1$, then $X$ has at least linear volume growth.
More precisely, for every $p\in X$, there exists a constant $C$ depending only on $m(B_{p}(1))$ and $N$ such that
$$m(B_{p}(r))\geq Cr.$$
\end{prop}

\begin{proof}
Let $x\in \partial B_{p}(1+r)$, then we have
$B_{p}(1)\subset B_{x}(2+r)\setminus B_{x}(r)$ and $B_{x}(2+r)\subset B_{p}(3+2r)$. Thus we have
\begin{align}\label{t1}
m(B_{p}(1))\leq m(B_{x}(2+r))-m(B_{x}(r)),
\end{align}
\begin{align}\label{t2}
m(B_{x}(2+r))\leq m(B_{p}(3+2r)).
\end{align}

On the other hand, by (\ref{B-G-2}), we have
\begin{align}\label{t3}
m(B_{x}(2+r))-m(B_{x}(r))\leq m(B_{x}(2+r))\frac{(2+r)^{N}-r^{N}}{(2+r)^{N}}.
\end{align}
Combining (\ref{t1}), (\ref{t2}) and (\ref{t3}), we obtain
$$m(B_{p}(3+2r))\geq m(B_{p}(1))\frac{(2+r)^{N}}{(2+r)^{N}-r^{N}}\geq Cr$$
and finish the proof.
\end{proof}

We remark that in the proof of Proposition \ref{infinity-volume}, we only use the Bishop-Gromov volume comparison estimate (\ref{B-G}), thus similar property also holds for noncompact $\textmd{MCP}(0,N)$ or $\textmd{CD}(0,N)$ spaces.

In the following we review some results on $\textmd{RCD}(0,N)$ spaces that will be used in this paper.
In fact similar results are valid for general $\textmd{RCD}^{*}(K,N)$ spaces, but we state them for $\textmd{RCD}(0,N)$ spaces for simplicity.

An $\textmd{RCD}(0,N)$ space always has the Sobolev-to-Lipschitz property, i.e. any $f\in W^{1,2}(X)$ with $|Df|\leq 1$ $m$-a.e. admits a $1$-Lipschitz representative (see \cite{AGS14} \cite{Gig13}).

On an $\textmd{RCD}(0,N)$ space we consider the following space of test functions:
$$\textmd{Test}(X):=\{f\in D(\Delta)\mid f,|Df|\in L^{\infty}(X), \Delta f\in W^{1,2}(X)\}.$$
$\textmd{Test}(X)$ is dense in $W^{1,2}(X)$.

On an $\textmd{RCD}(0,N)$ space $(X,d,m)$, $\mathcal{E}$ is a quadratic form.
By the theory of gradient flows of convex and lower semicontinuous functions on Hilbert spaces (see e.g. \cite{AGS08} for a comprehensive presentation),
the heat flow $h_{t}:L^{2}(X)\rightarrow L^{2}(X)$, $t\geq0$ is the unique family of maps such that for any $f\in L^{2}(X)$ the curve $t\mapsto h_{t}(f)\in L^{2}(X)$ is continuous on $[0,\infty)$, locally absolutely continuous on $(0,\infty)$, and fulfills $h_{0}(f)=f$, $h_{t}(f)\in D(\Delta)$ for every $t>0$, and
$$\frac{d}{dt}h_{t}(f)=\Delta h_{t}(f),\qquad \mathcal{L}^{1}\text{-a.e.} t > 0.$$
Some classical results are:
\begin{align}\label{5.24-1}
\frac{d}{dt} \parallel h_{t}(f)\parallel^{2}_{L^{2}}=-4\mathcal{E}(h_{t}(f)), \quad \forall t>0;
\end{align}
\begin{align}\label{5.24-2}
\parallel h_{t}(f)\parallel_{L^{2}}\leq \parallel f\parallel_{L^{2}}, \quad \forall t\geq0;
\end{align}
\begin{align}\label{5.24-3}
\mathcal{E}(h_{t}(f))\leq\frac{1}{4t}\parallel f\parallel_{L^{2}}^{2}, \quad \forall t>0.
\end{align}

On $\textmd{RCD}(0,N)$ spaces the following Bochner inequality holds (see \cite{EKS15} \cite{AMS14}):
for all $f,g\in \textmd{Test}(X)$ with $g\geq0$, we have
\begin{align}\label{2.14}
\frac{1}{2}\int\Delta g|Df|^{2}dm\geq\int g\bigl[\frac{(\Delta f)^{2}}{N}+\langle\nabla f,\nabla\Delta f\rangle\bigr]dm.
\end{align}

The following lemma can be found in \cite{GigPhi15} (see \cite{AMS14} \cite{MN14} for related results), which provides cut-off functions with quantitative estimates on $\textmd{RCD}(0,N)$.

\begin{lem}\label{lem5.13}
Suppose $(X,d,m)$ is an $\textmd{RCD}(0,N)$ space.
For every $r>0$ there exists a constant $C(r)>0$ such that the following holds.
Given $K\subset U\subset X$ with $K$ compact and $U$ open such that $\inf_{x\in K,y\in U^{c}}d(x,y)\geq r$,
there exists a test function $\chi$ with values in $[0,1]$, which is $1$ on $K$, with $\textmd{supp}(\chi)\subset\subset U$ and satisfying
$$\textmd{Lip}(\chi) + \parallel\Delta\chi\parallel_{L^{\infty}}\leq C(r).$$
\end{lem}

For an $\textmd{RCD}(0,N)$ space $(X,d,m)$, the notion of second order Sobolev space $W^{2,2}(X)$ can be introduced as in \cite{Gig14-2}.
Since $L^{2}(T^{*}X)$ is a Hilbert module, we can define the Hilbert tensor product $L^{2}((T^{*})^{\otimes2}X)$. See Section 1.5 in \cite{Gig14-2} for related notions.
\begin{defn}
A function $f \in W^{1,2}(X)$ is belong to $W^{2,2}(X)$ provided there is an element of $L^{2}((T^{\ast})^{\otimes2}X)$, called the Hessian of $f$ and denoted by $\textmd{Hess}(f)$, such that for any $g_{1},g_{2},h\in \textmd{Test}(X)$ it holds
\begin{align}\label{def-hess}
2\int h\textmd{Hess}(f)(\nabla g_{1}, \nabla g_{2})=&\int\bigl[-\langle \nabla f,\nabla g_{1}\rangle \textmd{div}(h\nabla g_{2})-\langle \nabla f,\nabla g_{2}\rangle \textmd{div}(h\nabla g_{1})\\
&-h\langle\nabla f,\nabla(\langle\nabla g_{1},\nabla g_{2}\rangle)\rangle\bigr]dm.\nonumber
\end{align}

$W^{2,2}(X)$ is equipped with the norm
$$\parallel f\parallel^{2}_{W^{2,2}}:=\int\bigl(|f|^{2}+|Df|^{2}+|\textmd{Hess}(f)|^{2}\bigr)dm,$$
which make it a separable Hilbert space.
\end{defn}
It's proved in \cite{Gig14-2} that $D(\Delta)\subset W^{2,2}(X)$ and thus $W^{2,2}(X)$ is dense in $W^{1,2}(X)$.

Hessian satisfies the chain rule and Leibniz rule as well as the locality property, see Propositions 3.3.20 to 3.3.24 in \cite{Gig14-2} for precise statements.

By the locality property of $\textmd{Hess}$, given an open subset $U$, we define $W^{2,2}(U)$ as the subspace of $W^{1,2}_{loc}(X)$ consists of functions $f$ for which there is $\textmd{Hess}(f)\in L^{2}((T^{*})^{\otimes2}X)$ such that (\ref{def-hess}) holds for any $g_{1}, g_{2}, h\in \textmd{Test}(X)$ with support in $U$.

Finally, let's recall a useful approximation lemma, which is well known to experts:
\begin{lem}\label{lem2.13}
If $f\in W^{1,2}(X)$ satisfies $supp(f)\subset U$, where $U$ is an open set, then there exists a sequence of test functions $f_{i}$ such that $\textmd{supp}(f_{i})\subset\subset U$, $\Delta f_{i}\in L^{\infty}(X)$ and $f_{i}$ converging to $f$ in $W^{1,2}(X)$.
\end{lem}

\subsection{Disintegration}\label{subsec-disintegration}

\begin{defn}
Suppose $(X,\Omega,\mu)$ and $(Y,\Sigma,\nu)$ are measure spaces, and $f:X\rightarrow Y$ is a measurable map.
A disintegration of $\mu$ over $\nu$ consistent with $f$ is a map $\rho:\Omega\times Y\rightarrow[0, \infty]$ such that
\begin{enumerate}
  \item $\rho_{y}(\cdot)$ is a measure on $(X,\Omega)$ for every $y\in Y$,
  \item $\rho_{\cdot}(B)$ is $\nu$-measurable for all $B\in\Omega$,
  \item the consistency condition
  $$\mu(B\cap f^{-1}(C))=\int_{C}\rho_{y}(B)\nu(dy)$$
  holds for all $B\in\Omega$, $C\in\Sigma$.
\end{enumerate}

We say that the disintegration is unique if $\rho_{1}, \rho_{2}$ are two consistent disintegrations
then $\rho_{1,y}(\cdot)=\rho_{2,y}(\cdot)$ for $\nu$-a.e. $y\in Y$.

A disintegration is strongly consistent with $f$ if $\rho_{y}(X\setminus f^{-1}(\{y\}))=0$ holds for $\nu$-a.e. $y$.
\end{defn}

We recall the following version of the disintegration theorem which can be found in \cite{Frem-4} (see 452G and 452I) or \cite{BC09} (see Theorem A.7).

\begin{thm}[Disintegration of measures]\label{disin-1}
Supppose $(X,\Omega,\mu)$ and $(Y,\Sigma,\nu)$ are measure spaces such that $(X,\Omega,\mu)$ is a countably generated Radon measure space and $(Y,\Sigma)$ is countably separated.
Suppose there is an inverse-measure-preserving map $f:X\rightarrow Y$,
then there exists a unique disintegration $y\mapsto\rho_{y}$ over $\nu$  strongly consistent with $f$.
\end{thm}

We recall that a measurable space $(X,\Omega)$ is countably separated if there is a countable set $\Upsilon\subset \Omega$ such that for any distinct $x, y\in X$ there is some $E\in \Upsilon$ such that $x\in E$ and $y\notin E$.

Let $(X,\Omega,\mu)$ and $(Y,\Sigma,\nu)$ be measure spaces, a map $f:X\rightarrow Y$ is called inverse-measure-preserving provided for any $C\in\Sigma$, it holds $f^{-1}(C)\in\Omega$ and $\nu(C)=\mu(f^{-1}(C))$.

Note that the $\rho_{y}$ in Theorem \ref{disin-1} are probability measures for $\nu$-a.e. $y$.

\section{Busemann function and disintegration revisited}\label{sec3}

In this section, the metric measure space $(X,d,m)$ is always assumed to be noncompact and satisfying the $\textmd{MCP}(K,N)$ condition.
Note that $(X,d)$ is locally compact.

\subsection{General properties of Busemann functions}\label{sec3.1}

Firstly, let's recall some classical notions about Busemann functions.

Given a geodesic ray $\gamma$ emitted from $p$, for any $t\geq0$, denote by $b_{t}(x):=t-d(x,\gamma_{t})$.
By the triangle inequality, we are easy to see
\begin{enumerate}
  \item given any $x\in X$, the function $t\mapsto b_{t}(x)$ is non-decreasing;
  \item $b_{t}(x)\leq d(x,p)$ for all $t\geq0$.
\end{enumerate}

We define the Busemann function associated to $\gamma$ as
$$b(x):=\lim_{t\rightarrow+\infty}b_{t}(x).$$
Note that the convergence is uniform on any given compact sets.

Since $b_{t}(x)$ are all $1$-Lipschitz functions, $b$ is also $1$-Lipschitz.

For any given $x\in X$, let $\gamma^{t,x}:[0,d(x,\gamma_{t})]\rightarrow X$ be a unit speed geodesic connecting $x$ to $\gamma_{t}$, where $t\geq0$.
By the properness of $X$, there is a sequence $\{t_{n}\}$, with $t_{n}\rightarrow\infty$, such that $\gamma^{t_{n},x}$ converge on compact sets to a geodesic ray $\gamma^{x}:[0,\infty)\rightarrow X$ with $\gamma^{x}(0)=x$.
Such a ray $\gamma^{x}$ is called a Busemann ray associated with $\gamma$.
We note that different choice of sequences $\{t_{n}\}$ may give different Busemann rays.

\begin{lem}\label{lem3.1}
For every $t\geq0$, we have
$$b(\gamma^{x}(t))=b(x)+t.$$
\end{lem}

\begin{proof}
Suppose $\gamma^{t_{n},x}:[0,d(x,\gamma_{t_{n}})]\rightarrow X$ converge uniformly on compact sets to $\gamma^{x}:[0,\infty)\rightarrow X$.
Then $\gamma^{t_{n},x}(t)$ converge to $\gamma^{x}(t)$.
Thus
\begin{align}
&b(\gamma^{x}(t))=\lim_{n\rightarrow\infty}(t_{n}-d(\gamma^{x}(t),\gamma(t_{n})))\nonumber\\
=&\lim_{n\rightarrow\infty}[t_{n}-d(\gamma^{t_{n},x}(t),\gamma(t_{n}))+
(d(\gamma^{t_{n},x}(t),\gamma(t_{n}))-d(\gamma^{x}(t),\gamma(t_{n})))]\nonumber\\
=&\lim_{n\rightarrow\infty}(t_{n}-d(x,\gamma(t_{n})))+t=b(x)+t,\nonumber
\end{align}
and we finish the proof.
\end{proof}

For any $s\in\mathbb{R}$, we denote $\Omega_{s}:=\{x\in X\mid b(x)> s\}$.
For $c<d$, we denote the set $\{x\in X\mid b(x)\in[c,d]\}$ by $\Omega_{[c,d]}$.

\begin{lem}\label{lem3.2}
For any $s\in\mathbb{R}$, $\Omega_{s}$ is path-connected.
More precisely, any two points $x,y\in \Omega_{s}$ can be connected by a Lipschitz curve $\sigma\subset \Omega_{s}$.
\end{lem}
\begin{proof}
We only need to prove that any $x\in \Omega_{s}$ can be connected to $\gamma_{s+1}$ by a Lipschitz curve $\sigma\subset \Omega_{s}$.
Firstly, $\gamma^{x}(t)\subset \Omega_{s}$ connects $x$ and $y=\gamma^{x}(1)$, with $b(y)\geq s+1$.
Since $b(y)=\lim_{t\rightarrow\infty}(t-d(y,\gamma_{t}))$, we find $\tilde{t}$ sufficiently large such that $b_{\tilde{t}}(y)=\tilde{t}-d(y,\gamma_{\tilde{t}})>s+\frac{1}{2}$.
Let $\tilde{\sigma}:[0,d(y,\gamma_{\tilde{t}})]\rightarrow X$ be a unit speed geodesic connecting $y$ to $\gamma_{\tilde{t}}$.
Then for any $r\in[0,d(y,\gamma_{\tilde{t}})]$, we have $b_{\tilde{t}}(\tilde{\sigma}(r))=b_{\tilde{t}}(y)+r$.
Thus $b(\tilde{\sigma}(r))\geq b_{\tilde{t}}(\tilde{\sigma}(r))>s+\frac{1}{2}$ for any $r\in[0,d(y,\gamma_{\tilde{t}})]$.
So $\tilde{\sigma}\subset \Omega_{s}$ connects $y$ and $\gamma_{\tilde{t}}$.
Finally, $\gamma$ itself is a curve connecting $\gamma_{\tilde{t}}$ and $\gamma_{s+1}$.
Glue the above geodesics to form a Lipschitz curve and we complete the proof.
\end{proof}

\begin{lem}\label{lem4.1}
Suppose $s_{1}<s_{2}$ and $x_{1}\in b^{-1}(s_{1})$, $x_{2}\in b^{-1}(s_{2})$, then $d(x_{1},x_{2})\geq s_{2}-s_{1}$.
\end{lem}
\begin{proof}
Suppose on the contrary, $d(x_{1},x_{2})< s_{2}-s_{1}$.
Since $b$ is $1$-Lipschitz, we have $s_{2}-s_{1}=b(x_{2})-b(x_{1})\leq d(x_{1},x_{2})< s_{2}-s_{1}$, which is a contradiction.
\end{proof}

On the other hand, suppose $s_{1}<s_{2}$ and $x\in b^{-1}(s_{1})$, since $\gamma^{x}([0,\infty))\cap b^{-1}(s_{2})=\gamma^{x}(s_{2}-s_{1})$ by Lemma \ref{lem3.1}, we have $\textmd{dist}(b^{-1}(s_{1}),b^{-1}(s_{2}))\leq s_{2}-s_{1}$.
Combining this with Lemma \ref{lem4.1}, we obtain

\begin{cor}\label{cor3.4}
Suppose $s_{1}<s_{2}$ and $b^{-1}(s_{1})\neq\emptyset$, then $\textmd{dist}(b^{-1}(s_{1}),b^{-1}(s_{2}))=s_{2}-s_{1}$.
\end{cor}

\begin{cor}\label{cor3.5}
Suppose $s_{1}<s_{2}$, $b^{-1}(s_{1})\neq\emptyset$ and $\textmd{diam}(b^{-1}(s_{2}))\leq T$, then $\textmd{diam}(b^{-1}(s_{1}))\leq T+2(s_{2}-s_{1})$ and $\textmd{diam}(\Omega_{[s_{1},s_{2}]})\leq T+2(s_{2}-s_{1})$.
\end{cor}

In Lemma 5.18 of \cite{Gig15}, it is proved that for a Busemann function $b$, $-b$ is $c$-concave, where $c=\frac{d^{2}}{2}$.
We extend this fact to the following proposition:

\begin{prop}\label{3.2}
For any $t>0$, the function $-tb$ is $c$-concave, where $c=\frac{d^{2}}{2}$.
Furthermore, for any $x, \bar{x}\in X$ satisfying $b(\bar{x})-b(x)=d(x,\bar{x})=t$, it holds $(x,\bar{x})\in \partial^{c}(-tb)$.
\end{prop}

\begin{proof}
For any $x, \bar{x}\in X$ satisfying $b(\bar{x})-b(x)=d(x,\bar{x})=t$, we have
\begin{align}\label{3.11}
(-tb)^{cc}(x)&=\inf_{y\in X}[\frac{d^{2}(x,y)}{2}-(-tb)^{c}(y)]\\
&\leq \frac{d^{2}(x,\bar{x})}{2}-(-tb)^{c}(\bar{x})\nonumber\\
&=\frac{t^{2}}{2}-\inf_{z\in X}[\frac{d^{2}(\bar{x},z)}{2}+tb(z)]\nonumber\\
&\leq \frac{t^{2}}{2}-\inf_{z\in X}[\frac{d^{2}(\bar{x},z)}{2}+tb(\bar{x})-td(\bar{x},z)]\nonumber\\
&= \frac{t^{2}}{2}-\inf_{z\in X}[\frac{d^{2}(\bar{x},z)}{2}+tb(x)+t^{2}-td(\bar{x},z)]\nonumber\\
&=-tb(x)-\inf_{z\in X}[\frac{d^{2}(\bar{x},z)}{2}-td(\bar{x},z)+\frac{t^{2}}{2}]\nonumber\\
&=-tb(x)-\frac{1}{2}\inf_{z\in X}[d(\bar{x},z)-t]^{2}\nonumber\\
&\leq-tb(x).\nonumber
\end{align}
On the other hand, by (\ref{2.2}), $(-tb)^{cc}\geq -tb$ always holds, thus $(-tb)^{cc}(x)=-tb(x)$, and all the inequality in (\ref{3.11}) must be equality.
In particular, we have
$$-tb(x)=\frac{d^{2}(x,\bar{x})}{2}-(-tb)^{c}(\bar{x}),$$
i.e. $(x,\bar{x})\in \partial^{c}(-tb)$.
\end{proof}

The main result of this section is that we have a monotonic property on the area of level sets of a Busemann function under suitable conditions, see Proposition \ref{prop3.19}, Corollary \ref{c-1} and similar properties in Section \ref{sec3.2}.
The tool we use is the general results established in \cite{BC13} and \cite{Ca14-2}.
In the following we will present how the notions and results in \cite{BC13} and \cite{Ca14-2} can be modified to our setting, and obtain some special properties on the transport set of Busemann function.

\begin{defn}
The set of couples moved by $b$ is defined to be:
\begin{align}
\Gamma:=\{(x, y)\in X\times X\mid b(y)-b(x)=d(x,y)\}.
\end{align}
\end{defn}


\begin{lem}\label{lem3.6}
Suppose $(x,y),(y,z)\in\Gamma$, then $(x,z)\in\Gamma$ and $d(x,z)=d(x,y)+d(y,z)$.
\end{lem}
\begin{proof}
Because $d(x,z)\leq d(x,y)+d(y,z)= b(z)-b(x)\leq d(x,z)$.
\end{proof}

\begin{lem}\label{lem3.7}
Let $(x,y)\in X\times X$ be an element of $\Gamma$.
Let $\sigma\in \textmd{Geo}(X)$ be such that $\sigma_{0}=x$ and $\sigma_{1}=y$.
Then
$(\sigma_{s},\sigma_{t})\in\Gamma$
for all $0\leq s\leq t\leq 1$.
\end{lem}
\begin{proof}
For $0\leq s\leq t\leq 1$, we have
\begin{align}
d(\sigma_{s},\sigma_{t})\geq &b(\sigma_{t})-b(\sigma_{s})
=(b(\sigma_{1})-b(\sigma_{0}))+(b(\sigma_{t})-b(\sigma_{1}))+ (b(\sigma_{0})-b(\sigma_{s}))\nonumber\\
\geq & d(\sigma_{0},\sigma_{1})-d(\sigma_{t},\sigma_{1})-d(\sigma_{0},\sigma_{s})=d(\sigma_{s},\sigma_{t}),\nonumber
\end{align}
the claim follows.
\end{proof}

It is then natural to consider the set of geodesics $G:=\{\sigma\in \textmd{Geo}(X)\mid(\sigma_{0},\sigma_{1})\in\Gamma\}$.

For any $x$ and a Busemann ray $\gamma^{x}:[0,\infty)\rightarrow X$, by Lemma \ref{lem3.1} we have $(\gamma^{x}_{s},\gamma^{x}_{t})\in\Gamma$ for any $s\leq t$.

\begin{defn}\label{def3.1}
We define the set of transport rays by
$$R=\Gamma\cup\Gamma^{-1},$$
where $\Gamma^{-1}:=\{(x, y)\in X\times X\mid (y, x)\in \Gamma\}$.
For fixed $x$, we use $\Gamma(x)$ to denote $\{y\in X\mid (x,y)\in \Gamma\}$, $\Gamma^{-1}(x)$ to denote $\{y\in X\mid (y,x)\in \Gamma\}$.
Let $R(x)=\Gamma(x)\cup\Gamma^{-1}(x)$.
\end{defn}

\begin{rem}
Since $b$ is $1$-Lipschitz, $\Gamma$, $\Gamma^{-1}$ and $R$ are all closed.
Moreover $\Gamma$, $\Gamma^{-1}$ and $R$ are all $\sigma$-compact because $X$ is proper.
\end{rem}

\subsection{RCD(0,N) case}\label{sec3.3}

In this subsection, we use the cost function $c=\frac{d^{2}}{2}$.
Firstly, we recall the following theorem from \cite{GRS13}:

\begin{thm}[Theorem 1.3 in \cite{GRS13}]\label{3.25}
Suppose $(X,d,m)$ is an $\textmd{RCD}^{*}(K,N)$ space and $\varphi: X\rightarrow \mathbb{R}$ a $c$-concave function.
Then for $m$-a.e. $x\in X$ there exists exactly one geodesic $\eta$ such that $\eta_{0}=x$ and $\eta_{1}\in \partial^{c}\varphi(x)$.
\end{thm}

Theorem \ref{3.25} is equivalent to the fact that on an $\textmd{RCD}^{*}(K,N)$ space $(X,d,m)$, for every $\mu,\nu\in \mathcal{P}_{2}(X)$ with $\mu\ll m$, there exists a unique plan $\pi\in\textmd{OptGeo}(\mu,\nu)$ and this $\pi$ is induced by a map and concentrated on a set of non-branching geodesics. 
See \cite{GRS13}. This generalizes the results in \cite{RaSt14} and \cite{Gig12}.

In this subsection, we always assume $(X,d,m)$ is noncompact and satisfies the $\textmd{RCD}^{*}(K,N)$ condition.

Recall Proposition \ref{3.2}, for any $n\in\mathbb{N}$, we apply Theorem \ref{3.25} to the function $-nb$.
Thus for any $n\in\mathbb{N}$, there exists a Borel set $B_{n}\subset X$ satisfying the following:
$m(B_{n})=0$, and if we denote by $A_{n}=X\setminus B_{n}$, then for any $x\in A_{n}$, there exists exactly one point $y\in X$ such that $b(y)-b(x)=d(x,y)=n$ and there is only one geodesic $\eta^{(x)}:[0,n]\rightarrow X$ such that $\eta^{(x)}(0)=x$ and $\eta^{(x)}(n)= y$.

Let $\mathcal{T}:=\bigcap_{n=1}^{\infty}A_{n}$, then
\begin{prop}\label{3.30}
$m(X\setminus \mathcal{T})=0$, and for any $x\in \mathcal{T}$, there exists exactly one geodesic $\eta^{(x)}:[0,\infty)\rightarrow \mathcal{T}$ such that $\eta^{(x)}_{0}=x$ and $b(\eta^{(x)}_{t})-b(x)=t$ for any $t\geq 0$.
Furthermore, if $\eta^{(x)}:[0,\infty)\rightarrow \mathcal{T}$ and $\eta^{(y)}:[0,\infty)\rightarrow \mathcal{T}$ satisfies $\eta^{(x)}(s)=\eta^{(y)}(t)$ for some $s\leq t$, then $\eta^{(x)}$ must coincide with $\eta^{(y)}|_{[t-s,\infty)}$.
\end{prop}
\begin{proof}
For any $x\in \mathcal{T}$, let $\gamma^{x}:[0,\infty)\rightarrow X$ be one of the Busemann rays, if we can prove that for any $t>0$, $\gamma^{x}(t)\in\mathcal{T}$, then by the definition of $\mathcal{T}$, it is easy to check that  $\gamma^{x}$ is the only Busemann rays start from $x$, and it is the $\eta^{(x)}$ we want to find.
Suppose $\gamma^{x}(m)\notin\mathcal{T}$ for some $m>0$, then $\gamma^{x}(m)\notin A_{n}$ for some $n$, thus there is two different geodesics $\sigma_{1},\sigma_{2}:[0,n]\rightarrow X$ such that $\sigma_{1}(0)=\sigma_{2}(0)=\gamma^{x}(m)$, $b(\sigma_{i}(n))-b(\sigma_{i}(0))=d(\sigma_{i}(0),\sigma_{i}(n))=n$, $i=1,2$.
By Lemmas \ref{lem3.1}, \ref{lem3.6}, \ref{lem3.7} and Proposition \ref{3.2}, $\gamma^{x}|_{[0,m]}$ and $\sigma_{i}$ can be glued to a longer geodesic $\tilde{\sigma}_{i}:[0,m+n]\rightarrow X$, with $(x,\tilde{\sigma}_{i}(m+n))\in \partial(-(m+n)b)$. This contradicts to $x\in\mathcal{T}$.
\end{proof}

We are easy to see $R$ is an equivalent relation on $\mathcal{T}$ by definition,
and for all $x\in\mathcal{T}$, $R(x)\cap \mathcal{T}$ forms a single geodesic ray.
In addition, for distinct $x, y\in \mathcal{T}$, either $R(x)=R(y)$, or $R(x)\cap R(y)\cap\mathcal{T}=\emptyset$.

Making use of the continuity and local compactness of geodesics as well as a special form of selection principle (see Corollary 2.7 in \cite{BC13}), one can prove the following:

\begin{prop}[See Proposition 4.4 in \cite{BC13}]\label{f-fcn}
There exists an $m$-measurable section
$f:\mathcal{T}\rightarrow\mathcal{T}$
for the equivalence relation $R$.
More precisely, there exists a saturated set $Z\subset \mathcal{T}$ such that $Z\in\mathcal{A}(X)$, $\mathcal{T}\setminus Z$ is $m$-negligible, and the section $f$ restricted on $Z$ is $\mathcal{A}(X)$-measurable.
\end{prop}

The proof of proposition \ref{f-fcn} is the same as Proposition 4.4 in \cite{BC13}.

Here we recall that a set $A\subset X$ is called to be saturated for the equivalence relation $E\subset X \times X$ if $A =\cup_{x\in A}E(x)$.
A map $f : X\rightarrow X$ is called a section of an equivalence relation $E$ if for any $x, y\in X$ it holds
$$(x,f (x))\in E,\quad\text{ and }\quad(x, y)\in E\Rightarrow f(x)=f(y). $$
A cross-section of the equivalence relation $E$ is a set $S\subset E$ such that the intersection of $S$ with each equivalence class is a singleton.

The set
$$D:=f(\mathcal{T})=\{x\in \mathcal{T}\mid d(x,f(x))=0\},$$
is obviously a cross-section.
$D$ inherit the subspace topology from $\mathcal{T}$.
In particular, $(D,\mathcal{B}(D))$ is countably separated.
By Lusin's theorem, we derive that there exists a $\sigma$-compact set $\mathcal{S}\subset D$ such that $m(\mathcal{T}\setminus f^{-1}(\mathcal{S})) = 0$.
Furthermore, $f$ restricted on $f^{-1}(S)$ is Borel.
The argument is as follows.
Let $C \subset X$ be any Borel set, then $f^{-1}(\mathcal{S}\cap C)=R^{-1}(\mathcal{S}\cap C)\cap\mathcal{T}$ and $f^{-1}(\mathcal{S}\setminus C)=R^{-1}(\mathcal{S}\setminus C)\cap\mathcal{T}$ are disjoint analytic sets.
By the first separation theorem for analytic sets (see Theorem 4.4.1 in \cite{Sri98}) there is a Borel set $B\subset X$ such that $f^{-1}(\mathcal{S}\cap C)\subset B$ and $f^{-1}(\mathcal{S}\setminus C)\cap B=\emptyset$.
Thus $f^{-1}(\mathcal{S}\cap C)= B\cap f^{-1}(\mathcal{S})$, i.e. $f$ is Borel on $f^{-1}(\mathcal{S})$.

Since $m$ is $\sigma$-finite, we fix a partition $\{\Gamma_{n}\}_{n\geq1}$ of $X$ into Borel sets of finite measure.
Let $\{\lambda_{n}\}_{n\geq1}$ be a sequence of positive real numbers such that
$\sum_{n\geq1}\lambda_{n}m(\Gamma_{n})=1$, and take
\begin{align}
\tilde{m}(B)=\sum_{n\geq1}\lambda_{n}m(\Gamma_{n}\cap B)
\end{align}
for $B\in \mathcal{B}(X)$.
Clearly, $\tilde{m}$ is a probability measure, $m$ and $\tilde{m}$ are absolutely continuous with respect to each other.
In particular, $m(B)=0\Leftrightarrow\tilde{m}(B)=0$ for $B\in\mathcal{B}(X)$.
Now the push forward measure $\tilde{\nu}= f_{*}\tilde{m}$ is well defined.
Obviously, $\tilde{\nu}$ concentrates on $D$, and $\tilde{\nu}(D\setminus\mathcal{S})=0$.

We apply Theorem \ref{disin-1} to $(f^{-1}(\mathcal{S}),\mathcal{B}(f^{-1}(\mathcal{S})),\tilde{m})$ and $(\mathcal{S},\mathcal{B}(\mathcal{S}),\tilde{\nu})$ and obtain a unique disintegration $y\mapsto\tilde{\rho}_{y}$ of $\tilde{m}$ over $\tilde{\nu}$ strongly consistent with  $f$ , i.e.
\begin{align}\label{3.1}
\tilde{m}(B\cap f^{-1}(C))=\int_{C}\tilde{\rho}_{y}(B)\tilde{\nu}(dy)
\end{align}
holds for every $B\in\mathcal{B}(f^{-1}(\mathcal{S}))$, $C\in\mathcal{B}(\mathcal{S})$.

We define $\rho:\mathcal{B}(f^{-1}(\mathcal{S}))\times \mathcal{S}\rightarrow[0,\infty]$ by
$$\rho_{y}(B)=\sum_{n\geq1}\lambda_{n}^{-1}\tilde{\rho}_{y}(\Gamma_{n}\cap B)$$
for every $B\in\mathcal{B}(f^{-1}(\mathcal{S}))$ and $y\in \mathcal{S}$, then we have
\begin{align}
m(B\cap f^{-1}(C))&=\sum_{n\geq1}\lambda_{n}^{-1}\tilde{m}(\Gamma_{n}\cap B\cap f^{-1}(C))\nonumber\\
&=\sum_{n\geq1}\lambda_{n}^{-1}\int_{C}\tilde{\rho}_{y}(\Gamma_{n}\cap B)\tilde{\nu}(dy)\nonumber\\
&=\int_{C}{\rho}_{y}(B)\tilde{\nu}(dy).\nonumber
\end{align}
Thus $y\mapsto{\rho}_{y}$ is a disintegration of $m$ over $\tilde{\nu}$ consistent with $f$.
It is easy to check $y\mapsto{\rho}_{y}$ is the unique strongly consistent disintegration w.r.t. $f$ over $\tilde{\nu}$.

Since $\tilde{\nu}(D\setminus \mathcal{S})=m(\mathcal{T}\setminus f^{-1}(\mathcal{S}))=0$, for any $B\subset \mathcal{T}$ $m$-measurable and $C\subset D$ $\tilde{\nu}$-measurable, it holds
\begin{align}\label{disintegration}
m(B\cap f^{-1}(C))=\int_{C}\rho_{y}(B)\tilde{\nu}(dy),
\end{align}
and for $\tilde{\nu}$-a.e.  $y\in D$, it holds
\begin{align}\label{3.3}
\rho_{y}(f^{-1}(D\setminus \{y\}))=0.
\end{align}

We define the ray map as follows:

\begin{defn}\label{def3.13}
Define the ray map $g: \textmd{Dom}(g)\subset D\times\mathbb{R}\rightarrow \mathcal{T}$ by the formula
\begin{align}
graph(g):=&\bigl\{(y, t, x)\in D\times\mathbb{R}\times \mathcal{T}\mid y\in D, t\in\mathbb{R}, x\in R(y), b(x)=t\bigr\}\nonumber
\end{align}
\end{defn}
Since $R(y)\cap \mathcal{T}$ is the single geodesic for $y\in D$, and the restriction of $b$ on $R(y)$ is a strictly monotony function, the set in the above definition is clearly the graph of some map $g$.

From the definition of the ray map, we immediately obtain the following properties:

\begin{prop}\label{p3.14}
The following properties hold:
\begin{enumerate}
  \item The ray map $g$ restricted on $\textmd{Dom}(g)\cap \mathcal{S}\times\mathbb{R}$ is Borel.
  \item For every fixed $y\in D$, the map $t\mapsto g(y,t)$ is $1$-Lipschitz $\Gamma$-order preserving.
  \item $(y,t)\mapsto g(y,t)$ is bijective on $\mathcal{T}$, and its inverse is
  $$x\mapsto g^{-1}(x)=(f(x),b(x)),$$
  where $f$ is the quotient map defined in Proposition \ref{f-fcn}.
\end{enumerate}
\end{prop}

We remark that in Section 4 of \cite{BC13}, there is a definition of ray map, and it has similar properties as in Proposition \ref{p3.14}.

\begin{defn}\label{def-flow1}
The ray map $g$ defines a flow on $\mathcal{T}$:
for $t\in \mathbb{R}^{+}$, we define
\begin{align}
F_{t}(x):=g(y,s+t)
\end{align}
for any $x=g(y,s)\in\mathcal{T}$.
\end{defn}
It is not hard to check $F_{t}$ is Borel for any $t\in \mathbb{R}^{+}$, and for every $t, s \in \mathbb{R}^{+}$ and $x\in\mathcal{T}$, it holds
\begin{align}
&F_{t+s}(x)=F_{t}(F_{s}(x)),\label{3.6}\\
&d(F_{s}(x),F_{t}(x))=|s-t|.\label{3.7}
\end{align}

Now we recall an important estimate from \cite{BC13} and \cite{Ca14-2}.
Let $\tilde{A}$ be a compact set contained in $\mathcal{T}$ with $m(\tilde{A})>0$, without loss of generality, we assume $\max_{x\in \tilde{A}} b(x)=0$.
For any $\delta>0$, define a Borel transport map $T$ on $\tilde{A}$ by
$$\tilde{A}\ni x\mapsto T(x)=F_{s(x)}(x),$$
where $s(x)=\delta-b(x)$.

Denote by $\mu=\frac{1}{m(\tilde{A})}m|_{\tilde{A}}$, $\nu=T_{*}\mu$.
Then $\omega=(Id,T)_{*}\mu$ be a transport plan between $\mu$ and $\nu$, and $\omega$ is concentrated on
the set $\Upsilon=\{(x,T(x))|x\in \tilde{A}\}\subset \Gamma$.
For any $\{(x_{i}, y_{i})\}_{i\leq N}\subset \Upsilon$, we have
\begin{align}
\sum_{i=1}^{N}d^{2}(x_{i},y_{i})=\sum_{i=1}^{N}|b(x_{i})-b(y_{i}))|^{2}=\sum_{i=1}^{N}|b(x_{i})-b(y_{i+1}))|^{2}\leq \sum_{i=1}^{N}d^{2}(x_{i},y_{i+1}),
\end{align}
that is, $\Upsilon$ is $c$-cyclically monotone.
Thus $\omega$ is an optimal transportation between $\mu$ and $\nu$.
By the results in \cite{GRS13}, $\omega$ is the unique optimal transportation, and for $\mu$-a.e $x$, the curve $t\mapsto (T_{t})_{*}\mu$, $t\in[0,1]$, is the unique geodesic connecting $\mu$ and $\nu$, where $T_{t}: \tilde{A}\rightarrow X$ is defined by $T_{t}(x)=F_{ts(x)}(x)$ for $x\in \tilde{A}$.
Denote by $\tilde{A}_{t}:=\{T_{t}(x)\mid x\in \tilde{A}\}$.
We can prove the following estimate:
\begin{align}\label{3.4}
m(\tilde{A}_{t})\geq (1-t)m(\tilde{A}) \min_{x\in \tilde{A}} \biggl\{\frac{s_{K}((1-t)d(x,T(x))/\sqrt{N-1})}{s_{K}(d(x,T(x))/\sqrt{N-1})}\biggr\}^{N-1}, \quad \forall t\in[0,1].
\end{align}

Because $(X,d,m)$ satisfies $\textmd{MCP}(K,N)$ condition and the essentially non-branching property, making use of these two properties, the proof of (\ref{3.4}) can be easily modified from Section 9 in \cite{BC13} or Section 4 in \cite{CaHu15}. 

By the inner regularity of Radon measures, it is not hard to see that the estimate (\ref{3.4}) still holds if  $\tilde{A}$ is only assume to be Borel and $m(\tilde{A})\in(0,\infty)$.

Under the assumption that $(X,d,m)$ is an $\textmd{RCD}^{*}(K,N)$ space, we can prove the absolute continuity of conditional measures:

\begin{thm}\label{thm3.15}
For $\tilde{\nu}$-a.e. $y\in \mathcal{S}$, the conditional measures $\rho_{y}$ are absolutely continuous w.r.t. $g(y,\cdot)_{*}\mathcal{L}^{1}$.
More precisely, there is some function $q(\cdot,\cdot): \textmd{Dom}(g)\cap \mathcal{S}\times\mathbb{R}\rightarrow[0,\infty)$
such that
\begin{align}\label{3.5}
m=g_{*}(q\tilde{\nu}\otimes \mathcal{L}^{1})
\end{align}
and
\begin{align}
\rho_{y}=g(y,\cdot)_{*}(q(y,\cdot)\mathcal{L}^{1})
\end{align}
for $\tilde{\nu}$-a.e. $y\in \mathcal{S}$.
\end{thm}

The proof of Theorem \ref{thm3.15} needs the estimate (\ref{3.4}), see Section 6.1 in \cite{Ca14-2} for details.

Furthermore we can obtain the following estimate for the function $q$:

\begin{thm}\label{thm3.16}
If $(X,d,m)$ is a noncompact $\textmd{RCD}^{*}(K,N)$ space, then
\begin{align}\label{3.14}
\biggl[\frac{s_{K}((\sigma_{+}-t)/\sqrt{N-1})}{s_{K}((\sigma_{+}-s)/\sqrt{N-1})}\biggr]^{N-1}
\leq \frac{q(y, t)}{q(y, s)}
\leq \biggl[\frac{s_{K}((t-\sigma_{-})/\sqrt{N-1})}{s_{K}((s-\sigma_{-})/\sqrt{N-1})}\biggr]^{N-1},
\end{align}
holds for $\tilde{\nu}$-a.e. $y\in \mathcal{S}$ and $\sigma_{-}< s\leq t <\sigma_{+}$ such that $(\sigma_{-},\sigma_{+})\subset \textmd{Dom}(g(y,\cdot))$.
\end{thm}

Theorem \ref{thm3.16} can be proved by using the disintegration formula (\ref{disintegration}) to localize estimates of the form (\ref{3.4}).
See Section 9 in \cite{BC13} or Appendix in \cite{Ca14-2} for more details of the proof.
Note that in the statement of Theorem \ref{thm3.15}, the function $q$ is just a measurable function, while in the statement of Theorem \ref{thm3.16}, the $q$ means one of its representative.
In the following, $q$ will always be a representative satisying (\ref{3.14}).
Note that for $\tilde{\nu}$-a.e. $y\in \mathcal{S}$, $t\mapsto q(y,t)$ is continuous in the $t$ direction.

In the remaining of this subsection, we assume $(X,d,m)$ is a noncompact $\textmd{RCD}(0,N)$ space.
Note that $\sigma_{+}$ can be taken to be any large number converging to $+\infty$.
Note that in this case
\begin{align}
\lim_{\sigma_{+}\rightarrow+\infty}\frac{s_{0}((\sigma_{+}-t)/\sqrt{N-1})}{s_{0}((\sigma_{+}-s)/\sqrt{N-1})} =\lim_{\sigma_{+}\rightarrow+\infty}\frac{\sigma_{+}-t}{\sigma_{+}-s}=1,\nonumber
\end{align}
hence we obtain

\begin{cor}\label{cor3.17}
If $(X,d,m)$ is a noncompact $\textmd{RCD}(0,N)$ space, then
\begin{align}\label{3.9}
q(y, t)\geq q(y, s)
\end{align}
holds for $\tilde{\nu}$-a.e. $y$ and $s\leq t$ such that $[s,t]\subset \textmd{Dom}(g(y,\cdot))$.
\end{cor}

Let $K\subset b^{-1}((-\infty,r_{0}])$ be a compact set.
Denote by
\begin{align}\label{symb-1}
\Xi(K):=\bigcup_{y\in K}R(y),
\end{align}
\begin{align}\label{symb-2}
\Xi_{[s,t]}(K):=\Xi(K)\cap b^{-1}([s,t]),
\end{align}
\begin{align}\label{symb-3}
\Xi_{s}(K):=\Xi(K)\cap b^{-1}(s),
\end{align}
\begin{align}\label{symb-4}
\mathcal{S}(K):=\Xi(K)\cap\mathcal{S}.
\end{align}
It is easy to check that $\Xi(K)$, $\Xi_{[s,t]}(K)$ and $\Xi_{s}(K)$ are all closed sets provided $s\geq r_{0}$.

Let $K\subset b^{-1}((-\infty,r_{0}])$ be a compact set contained in $f^{-1}(\mathcal{S})$.
By the construction of $\mathcal{T}$ and $\mathcal{S}$, (\ref{disintegration}) and Theorem \ref{thm3.15}, for any $r_{2}>r_{1}\geq r_{0}$, we have
\begin{align}\label{3.8}
m(\Xi_{[r_{1},r_{2}]}(K))=\int_{\mathcal{S}(K)}\biggl(\int_{r_{1}}^{r_{2}}q(y,s)ds\biggr)\tilde{\nu}(dy).
\end{align}
In fact, (\ref{3.8}) holds for any compact set $K\subset b^{-1}((-\infty,r_{0}])$.

By Corollary \ref{cor3.17} and (\ref{3.8}), we are not hard to obtain:
\begin{prop}\label{prop3.19}
Suppose $(X,d,m)$ is a noncompact $\textmd{RCD}(0,N)$ space.
Let $K\subset b^{-1}((-\infty,r_{0}])$ be a compact set, then
\begin{align}\label{3.17-2}
\frac{m(\Xi_{[r_{1},r_{2}]}(K))}{r_{2}-r_{1}}\leq\frac{m(\Xi_{[r_{2},r_{3}]}(K))}{r_{3}-r_{2}},
\end{align}
\begin{align}\label{3.17-3}
\frac{m(\Xi_{[r_{1},r_{2}]}(K))}{r_{2}-r_{1}}\leq\frac{m(\Xi_{[r_{1},r_{3}]}(K))}{r_{3}-r_{1}}
\end{align}
hold for any $r_{3}> r_{2}>r_{1}\geq r_{0}$.
\end{prop}

Now we can define the `codimension 1' volume $m_{-1}$ of $\Xi_{s}(K)$ to be
$$m_{-1}(\Xi_{s}(K)):=\lim_{r\downarrow s}\frac{m(\Xi_{[s,r]}(K))}{r-s}.$$
By (\ref{3.17-3}), $m_{-1}(\Xi_{s}(K))$ is well defined for $s\in[r_{0},\infty)$.
By Theorem \ref{thm3.16}, the map $s\mapsto\int_{\mathcal{S}(K)}q(y,s)\tilde{\nu}(dy)$ is continuous.
By Fubini's Theorem, we have
$$m_{-1}(\Xi_{s}(K))=\int_{\mathcal{S}(K)}q(y,s)\tilde{\nu}(dy).$$

By Proposition \ref{prop3.19}, we have:
\begin{cor}\label{c-1}
Suppose $K\subset b^{-1}((-\infty,r_{0}])$ is a compact set, then
\begin{align}\label{3.17-1}
m_{-1}(\Xi_{r_{1}}(K))\leq\frac{m(\Xi_{[r_{1},r_{2}]}(K))}{r_{2}-r_{1}}\leq m_{-1}(\Xi_{r_{2}}(K))
\end{align}
holds for any $r_{2}>r_{1}\geq r_{0}$.
\end{cor}

We can also define the `codimension 1' volume of $b^{-1}(s)$ to be
$$m_{-1}(b^{-1}(s)):=\limsup_{K} m_{-1}(\Xi_{s}(K)),$$
where the supremum is taken with respect to all compact $K\subset b^{-1}((-\infty,s])$.


\subsection{Non-branching MCP(0,N) case}\label{sec3.2}

In this subsection, we always assume $(X,d,m)$ is noncompact, non-branching and satisfies the $\textmd{MCP}(K,N)$ condition.

By the non-branching assumption, we have:
\begin{lem}\label{lem3.8}
Suppose $\sigma\in \textmd{Geo}(X)$ satisfies $\sigma_{0}=x$, $\sigma_{1}=y$, $(x,y)\in\Gamma$, then $\sigma$ is contained in a geodesic ray $\bar{\gamma}:[0,\infty)\rightarrow X$, with $\bar{\gamma}(0)=x$ and
$(\bar{\gamma}_{s},\bar{\gamma}_{t})\in\Gamma$ for every
$t\geq s\geq0$.
\end{lem}

\begin{proof}
Without loss of generality, we assume $\sigma$ has unit length.
Denote by $z=\sigma_{\frac{1}{2}}$, and $\gamma^{z}:[0,\infty)\rightarrow X$ a Busemann ray.
By Lemmas \ref{lem3.1}, \ref{lem3.6} and \ref{lem3.7}, we know
$\sigma\mid_{[0,\frac{1}{2}]}$ and $\gamma^{z}\mid_{[0,\frac{1}{2}]}$ glue to be a geodesic.
Combining this with the non-branching assumption, we obtain that $\sigma\mid_{[\frac{1}{2},1]}$ and $\gamma^{z}\mid_{[0,\frac{1}{2}]}$ coincide, while $\sigma\mid_{[0,\frac{1}{2}]}$ and $\gamma^{z}$ form a longer geodesic ray $\bar{\gamma}:[0,\infty)\rightarrow X$ such that
$b(\bar{\gamma}_{t})-b(\bar{\gamma}_{s})=d(\bar{\gamma}_{s},\bar{\gamma}_{t})$ for any $t\geq s\geq0$.
\end{proof}

\begin{defn}\label{def3.2}
The set of initial points is defined to be
$$a^{+}:=\{z\in X\mid \nexists x\in X\text{ such that }(x,z)\in\Gamma, d(x, z) > 0\}.$$
Define the transport set to be $\mathcal{T}:= X\setminus a^{+}$.
\end{defn}
\begin{rem}
$\mathcal{T}$ is analytic, because
$\mathcal{T}=P_{2}(\Gamma\cap\{(x, z)\in X\times X\mid d(x, z) > 0\})$.

Recall that we also use the same notation $\mathcal{T}$ in Section \ref{sec3.3}. Of course they have different meaning in the two subsection. But the two $\mathcal{T}$ play similar roles, so we use the same symbol for convenience.
\end{rem}

Under the non-branching assumption and $\textmd{MCP}(K,N)$ condition, we can prove:
\begin{prop}\label{3.12}
$a^{+}$ is $m$-negligible.
\end{prop}
The proof can be found in Section 9 of \cite{BC13}.
We give a sketch on the idea of its proof here for completeness.
Suppose on the contrary, $m(a^{+})>0$, then by a standard Borel selection argument and making use of the inner regularity of the measure, we find a compact set $\tilde{A}\subset a^{+}$ with $m(\tilde{A})>0$, and a continuous map $T:\tilde{A}\rightarrow \mathcal{T}$ such that $(x,T(x))\in\Gamma$ and $b(T(x))=c+1$, where $c:=\max_{x\in \tilde{A}}\{b(x)\}$.
Denote by $\mu=\frac{1}{m(\tilde{A})}m|_{\tilde{A}}$, $\nu=T_{*}\mu$. For any $x\in \tilde{A}$, let $\eta^{(x)}_{t}:[0,1]\rightarrow X$ be the geodesic such that $\eta^{(x)}_{0}=x$, $\eta^{(x)}_{1}=T(x)$.
Denote by $\tilde{A}_{t}:=\{e_{t}(\eta^{(x)})\mid x\in \tilde{A}\}$.
By the non-branching assumption and $\tilde{A}\subset a^{+}$, it is easy to prove $\tilde{A}_{t}\cap \tilde{A}_{s}=\emptyset$ provided $s\neq t$.
Since all the $\tilde{A}_{t}$'s are contained in a bounded set $K\subset X$, we have $\sum_{t\in[0,1]} m(\tilde{A}_{t})\leq m(K)<\infty$.
On the other hand, estimates as in (\ref{3.4}) holds, hence $m(\tilde{A}_{t})>0$ for every $t\in[0,1)$,
which implies $\sum_{t\in[0,1]} m(\tilde{A}_{t})=\infty$, and we get a contradiction.

We are easy to see $R$ is an equivalent relation on $\mathcal{T}$ by definition,
and for all $x\in\mathcal{T}$, $R(x)$ forms a single geodesic ray (see Lemma \ref{lem3.8}).
In addition, for distinct $x, y\in \mathcal{T}$, either $R(x)=R(y)$, or $R(x)\cap R(y)\cap\mathcal{T}=\emptyset$.

We can also prove the existence of an $m$-measurable section $f:\mathcal{T}\rightarrow\mathcal{T}$ as in Proposition \ref{f-fcn}, and then obtain a unique disintegration $y\mapsto\tilde{\rho}_{y}$ of $\tilde{m}$ over $\tilde{\nu}$ strongly consistent with  $f$  as in (\ref{disintegration}) (\ref{3.3}).
Under the non-branching assumption and $\textmd{MCP}(K,N)$ condition, estimates like (\ref{3.4}) holds, see Section 9 in \cite{BC13}.
Properties as in Theorems \ref{thm3.15}, \ref{thm3.16} are still true. We omit the details here. See also \cite{BC13}.

Suppose now $(X,d,m)$ is a noncompact non-branching $\textmd{MCP}(0,N)$ space, then Corollary \ref{cor3.17} holds on $X$.
We use notations similar to those from (\ref{symb-1}) to (\ref{symb-4}) and define `codimension 1' volume $m_{-1}$.
The inequalities in Proposition \ref{prop3.19} and Corollary \ref{c-1} still hold in this case.

\section{Noncompact \textmd{RCD}(0,N) spaces with minimal volume growth}
Recall that a noncompact $\textmd{MCP}(0,N)$ space has at least linear volume growth.
It is a natural problem to investigate those noncompact $\textmd{MCP}(0,N)$ spaces $(X,d,m)$ satisfying
\begin{align}\label{min-vol}
\limsup_{R\rightarrow\infty}\frac{m(B_{p}(R))}{R}=V_{0}<\infty.
\end{align}
We call $X$ has minimal volume growth if (\ref{min-vol}) holds.

In this section, the metric measure space $(X,d,m)$ is assumed to be either a noncompact $\textmd{RCD}(0,N)$ space satisfying (\ref{min-vol}) or a noncompact, non-branching space satisfying the $\textmd{MCP}(0,N)$ condition and (\ref{min-vol}).

\begin{prop}\label{c-2}
Suppose $(X,d,m)$ is a noncompact $\textmd{RCD}(0,N)$ space satisfying (\ref{min-vol}), then the level sets $b^{-1}(r)$ have finite `codimension 1' volume $m_{-1}$.
Furthermore, for any $r_{2}>r_{1}$, we have
\begin{align}\label{4.2}
m_{-1}(b^{-1}(r_{1}))\leq m_{-1}(b^{-1}(r_{2}))\leq V_{0}.
\end{align}
The same property holds for noncompact non-branching  $\textmd{MCP}(0,N)$ spaces satisfying (\ref{min-vol}).
\end{prop}

\begin{proof}
Assume that $(X,d,m)$ is a noncompact $\textmd{RCD}(0,N)$ space satisfying (\ref{min-vol}).
Let $K\subset b^{-1}((-\infty,r_{1}])$ be any compact set, denote by $\bar{r}=\max_{x\in\Xi_{r_{1}}(K)}d(p,x)$, then by Corollary \ref{cor3.4}, $\Xi_{[r_{1},r_{2}]}(K)\subset m(B_{p}(\bar{r}+r_{2}-r_{1}))$, and then by (\ref{3.17-1})
\begin{align}
m_{-1}(\Xi_{r_{1}}(K))\leq\frac{m(\Xi_{[r_{1},r_{2}]}(K))}{r_{2}-r_{1}}
\leq\frac{m(B_{p}(\bar{r}+r_{2}-r_{1}))}{r_{2}-r_{1}}
\end{align}
holds for every $r_{2}>r_{1}$.
Let $r_{2}\rightarrow\infty$, by (\ref{min-vol}), we obtain
\begin{align}
m_{-1}(\Xi_{r_{1}}(K))\leq V_{0}.
\end{align}
By the arbitrariness of $K\subset b^{-1}((-\infty,r_{1}])$,
\begin{align}
m_{-1}(b^{-1}(r_{1}))\leq V_{0}
\end{align}
holds for any $r_{1}$.
By (\ref{3.17-1}) again, we have
\begin{align}
m_{-1}(b^{-1}(r_{1}))\leq
m_{-1}(b^{-1}(r_{2}))
\end{align}
for any $r_{2}>r_{1}$.

Using the similar properties in Section \ref{sec3.2}, we can handle the non-branching $\textmd{MCP}(0,N)$ case.
\end{proof}

By (\ref{4.2}), the limit
\begin{align}
V_{\infty}:=\lim_{r\rightarrow\infty}m_{-1}(b^{-1}(r))\nonumber
\end{align}
exists and $0<V_{\infty}\leq V_{0}$.

The following theorem generalizes Theorem 19 of \cite{Sor98} to $\textmd{RCD}(0,N)$ spaces, and our argument is adapted from that of \cite{Sor98}.

\begin{thm}[=Theorem \ref{thm1.3}]\label{compact-level-set}
Suppose $(X,d,m)$ is a noncompact $\textmd{RCD}(0,N)$ space satisfying the minimal volume growth condition (\ref{min-vol}), and $b$ is the Busemann function associated to a geodesic ray $\gamma$.
Then the diameter of the level sets $b^{-1}(r)$ grow at most linearly.
More precisely, we have
\begin{align}\label{4.7}
\limsup_{R\rightarrow+\infty}\frac{\textmd{diam}(b^{-1}(R))}{R}\leq C_{0}\leq2,
\end{align}
where the diameter of $b^{-1}(R)$ is computed with respect to the distance $d$.
In particular, $b^{-1}(r)$ is compact.
\end{thm}

\begin{proof}
We argue by contradiction. Suppose (\ref{4.7}) does not hold, then there is a constant $c\in(0,\frac{1}{10})$ such that
\begin{align}\label{4.7-1}
\limsup_{R\rightarrow+\infty}\frac{\textmd{diam}(b^{-1}(R))}{R}\geq2+8c.
\end{align}

By Proposition \ref{c-2} and the definition of $V_{\infty}$, we can find $r_{0}$ and a compact set $K\subset b^{-1}((-\infty,r_{0}])\cap\mathcal{T}$ such that

\begin{align}\label{assump4.1}
\frac{V_{\infty}}{V_{r_{0}}}\leq\bigl(\frac{c}{2+5c}\bigr)^{N}+1,
\end{align}
where
$$
V_{r_{0}}:=m_{-1}(\Xi_{r_{0}}(K)).
$$
Fix an $x\in \Xi_{r_{0}}(K)$. For any $r>0$, take $R=(1+2c)r+\textmd{diam}(\Xi_{r_{0}}(K))$, we claim that there exists $T>0$ such that
\begin{align}\label{4.8}
b^{-1}(r+r_{0})\subset B_{x}(R)
\end{align}
for every $r> T$.
It is easy to see that if the claim holds, then it contradicts to (\ref{4.7-1}).
Thus (\ref{4.7}) can be proved.
Furthermore, by Corollary \ref{cor3.5}, the diameter of $b^{-1}(s)$ is finite provided $\textmd{diam}(b^{-1}(t))<\infty$ for some $t>s$, thus for any $r$, $\textmd{diam}(b^{-1}(r))$ is finite and hence $b^{-1}(r)$ is a compact subset of $X$.

Suppose (\ref{4.8}) does not hold, then there exist sequences of $\{r_{i}\}\subset\mathbb{R}^{+}$, $\{q_{i}\}\subset X$ such that $r_{i}\rightarrow\infty$, $b(q_{i})=r_{0}+r_{i}$, $d(x,q_{i})>(1+2c)r_{i}+\textmd{diam}(\Xi_{r_{0}}(K))$.
By Corollary \ref{cor3.4}, $\Xi_{[r_{0}+(1-c)r,r_{0}+(1+c)r]}(K)\subset B_{x}(R-cr)$ holds for any $r>0$, hence $d(q_{i},\Xi_{[r_{0}+(1-c)r_{i},r_{0}+(1+c)r_{i}]}(K))> cr_{i}$ and thus $d(q_{i},\Xi(K))> cr_{i}$ for every $i$.

By Lemma \ref{lem3.2}, for every $i$ there exists a curve $\sigma_{i}:[0,1]\rightarrow b^{-1}([r_{0}+r_{i},\infty))$ such that $\sigma_{i}(0)=q_{i}$ and $\sigma_{i}(1)\in\Xi_{r_{0}+ r_{i}}(K)$.
Since
$$d(\sigma_{i}(0),\Xi(K))-c(b(\sigma_{i}(0))-r_{0})>0,$$
and
$$d(\sigma_{i}(1),\Xi(K))-c(b(\sigma_{i}(1))-r_{0})=-cr_{i}<0,$$
there exists $p_{i}=\sigma_{i}(t_{i})$ with $t_{i}\in(0,1)$ such that $d(p_{i},\Xi(K))=c(b(p_{i})-r_{0})$.
Denote by $h_{i}=b(p_{i})-r_{0}$.
Obviously $h_{i}\rightarrow\infty$.

Denote by $\bar{r}=\inf_{y\in K}b(y)$.
Take $\bar{R}_{i}=(2+3c)h_{i}+\textmd{diam}(K)+2(r_{0}-\bar{r})$, then it is easy to check $\Xi_{[b(p_{i})-ch_{i},b(p_{i})+ch_{i}]}(K)\subset B_{p_{i}}(\bar{R}_{i})$.
Since the Bishop-Gromov volume comparison estimate (\ref{B-G}) holds on $X$, we have
\begin{align}
m(B_{p_{i}}(ch_{i}))&\geq\biggl(\frac{ch_{i}}{\bar{R}_{i}}\biggr)^{N}m(B_{p_{i}}(\bar{R}_{i}))
\geq\biggl(\frac{ch_{i}}{\bar{R}_{i}}\biggr)^{N}m(\Xi_{[b(p_{i})-ch_{i},b(p_{i})+ch_{i}]}(K))\\
&\geq\biggl(\frac{ch_{i}}{(2+3c)h_{i}+\textmd{diam}(K)+2(r_{0}-\bar{r})}\biggr)^{N}2ch_{i}V_{r_{0}}.\nonumber
\end{align}

Hence if $i$ large enough, we have
\begin{align}
m(B_{p_{i}}(ch_{i}))\geq\biggl(\frac{c}{2+4c}\biggr)^{N}2ch_{i}V_{r_{0}}>0.
\end{align}

It's easy to see that $B_{p_{i}}(ch_{i})\subset b^{-1}([b(p_{i})-ch_{i},b(p_{i})+ch_{i}])$.
Thus the sets $\Upsilon_{i,k}:=\Xi_{(b(p_{i})+(2k-1)ch_{i},b(p_{i})+(2k+1)ch_{i})}(B_{p_{i}}(ch_{i}))$, $k=1,2,\ldots$, are disjoint to each other.
By (\ref{3.17-2}), we have $m(\Upsilon_{i,k})\geq m(B_{p_{i}}(ch_{i}))$.
Hence we have
\begin{align}
m(\Xi_{[b(p_{i})+ch_{i},b(p_{i})+ch_{i}+t]}(B_{p_{i}}(ch_{i}))\geq \bigl(\frac{c}{2+4c}\bigr)^{N}(t-2ch_{i})V_{r_{0}}.
\end{align}
Note that we have $B_{p_{i}}(ch_{i})\cap \Xi(K)=\emptyset$, we know $\Xi(B_{p_{i}}(ch_{i}))\cap \Xi(K)$ is contained in the $m$-negligible set $X\setminus\mathcal{T}$.
Thus for any $r\gg ch_{i}$, we are easy to derive
\begin{align}\label{lower}
m(b^{-1}([r_{0}+(1+c)h_{i},r_{0}+(1+c)h_{i}+r]))\geq \bigl(\frac{c}{2+4c}\bigr)^{N}(r-2ch_{i})V_{r_{0}}+rV_{r_{0}}.
\end{align}
Multiply both sides of (\ref{lower}) by $\frac{1}{r}$, and then let $r\rightarrow\infty$.
By the definition of $V_{\infty}$ and (\ref{3.17-2}), we have $V_{\infty}\geq \bigl(\frac{c}{2+4c}\bigr)^{N}V_{r_{0}}+V_{r_{0}}$, which contradicts to (\ref{assump4.1}).
Thus we have completed the proof.
\end{proof}

Using the similar properties in Section \ref{sec3.2}, we can prove Theorem \ref{thm1.1} with the same proof.

\section{Noncompact \textmd{RCD}(0,N) spaces with strongly minimal volume growth}\label{sec5}

\subsection{Strongly minimal volume growth}

\begin{defn}\label{def5.1}
Suppose $(X,d,m)$ is a noncompact, non-branching $\textmd{MCP}(0,N)$ or a noncompact $\textmd{RCD}(0,N)$ space,
$b$ is the Busemann function associated to a geodesic ray $\gamma$.
We say $(X,d,m)$ has strongly minimal volume growth if
\begin{align}\label{str-min-vol}
\lim_{R\rightarrow\infty}\frac{m(B_{p}(R))}{R}=m_{-1}(\Xi_{r_{0}}(K))<\infty,
\end{align}
holds for some $r_{0}\in\mathbb{R}$ and a compact set $K\subset b^{-1}((-\infty,r_{0}])$.
(See Section \ref{sec3} for the definition of $\Xi_{r_{0}}(K)$.)
\end{defn}

Note that if (\ref{def5.1}) holds, then by Theorem \ref{thm1.3} or \ref{thm1.1}, every level set $b^{-1}(r)$ is compact.

In the remaining part of this paper, we will prove Theorem \ref{thm1.2}.
From now on, the metric measure space $(X,d,m)$ is always assumed to be noncompact and satisfying the $\textmd{RCD}(0,N)$ condition as well as (\ref{str-min-vol}).
We assume $r_{0}=0$ in (\ref{str-min-vol}) for convenience.

\begin{prop}\label{p5.2}
Suppose $(X,d,m)$ is a noncompact $\textmd{RCD}^{*}(0,N)$ space and satisfies (\ref{str-min-vol}), then $b^{-1}((0,\infty))\subset\Xi(K)$ and
\begin{align}\label{5.2}
    \frac{m(b^{-1}([r_{1},r_{2}]))}{r_{2}-r_{1}}=m_{-1}(\Xi_{0}(K))
\end{align}
for any $r_{2}>r_{1}\geq 0$.
\end{prop}

\begin{proof}
We first prove $b^{-1}((0,\infty))\subset\Xi(K)$.
Suppose on the contrary, there is some $z\in b^{-1}((0,\infty))\setminus \Xi(K)$.
Since $\Xi(K)$ is a closed set and $b^{-1}((0,\infty))$ is open, there is a small $\delta>0$ such that $B_{z}(\delta)\subset b^{-1}((0,\infty))\setminus \Xi(K)$.
Denote by $t=\min_{y\in K}\{b(y)\}$ and $r_{1}=b(z)+\delta$.
For any $r_{2}>r_{1}+1$, take $R=\max\{d(z,p)+r_{2}-b(z)+\delta,d(p,K)+\textmd{diam}(K)+r_{2}-t\}$.
It is easy to see
$B_{z}(\delta)\subset b^{-1}((-\infty,r_{1}])$
and
$\Xi_{[r_{1},r_{2}]}(B_{z}(\delta))\cup\Xi_{[r_{1},r_{2}]}(K)\subset B_{p}(R)$.

By Proposition \ref{prop3.19} and Corollary \ref{c-1}, we have
\begin{align}
m(\Xi_{[r_{1},r_{2}]}(B_{z}(\delta)))\geq&(r_{2}-r_{1})m(\Xi_{[r_{1},r_{1}+1]}(B_{z}(\delta)))>0,
\end{align}
and
\begin{align}
m(\Xi_{[r_{1},r_{2}]}(K))
\geq(r_{2}-r_{1})m_{-1}(\Xi_{r_{1}}(K))\geq(r_{2}-r_{1})m_{-1}(\Xi_{0}(K)).
\end{align}
On the other hand, since $B_{z}(\delta)\subset b^{-1}((0,\infty))\setminus \Xi(K)$, it is easy to see $\Xi(B_{z}(\delta))\cap\Xi(K)\cap b^{-1}((0,\infty))\subset X\setminus \mathcal{T}$, which is $m$-negligible.
Therefore,
\begin{align}
\frac{m(B_{p}(R))}{R}&\geq\frac{m(\Xi_{[r_{1},r_{2}]}(B_{z}(\delta))\cup\Xi_{[r_{1},r_{2}]}(K))}{R}\\ &\geq\frac{(r_{2}-r_{1})}{R}\bigl[m(\Xi_{[r_{1},r_{1}+1]}(B_{z}(\delta)))+m_{-1}(\Xi_{0}(K))\bigr].\nonumber
\end{align}
Let $r_{2}\rightarrow\infty$, we obtain
$$\lim_{R\rightarrow\infty}\frac{m(B_{x_{0}}(R))}{R}>m_{-1}(\Xi_{0}(K)),$$
contradicting to (\ref{str-min-vol}). This proves $b^{-1}((0,\infty))\subset\Xi(K)$.

Now we prove (\ref{5.2}). Suppose (\ref{5.2}) does not hold for some $r_{2}>r_{1}\geq0$, then
$m_{-1}(b^{-1}(r_{2}))>m_{-1}(\Xi_{0}(K))$.
For any $r_{3}>r_{2}$, take $R=d(p,b^{-1}(r_{2}))+\textmd{diam}(b^{-1}(r_{2}))+r_{3}-r_{2}$, then we have
$\Xi_{[r_{2},r_{3}]}(b^{-1}(r_{2}))\subset B_{p}(R)$.

By Corollary \ref{c-1}, we have
\begin{align}
\frac{m(B_{p}(R))}{R}&\geq\frac{m(\Xi_{[r_{2},r_{3}]}(b^{-1}(r_{2})))}{R}\geq\frac{(r_{3}-r_{2})}{R}m_{-1}(b^{-1}(r_{2})).\nonumber
\end{align}
Let $r_{3}\rightarrow\infty$, we obtain
$$\lim_{R\rightarrow\infty}\frac{m(B_{p}(R))}{R}\geq m_{-1}(b^{-1}(r_{2}))>m_{-1}(\Xi_{0}(K)),$$
contradicting to (\ref{str-min-vol}).
The proof is completed.
\end{proof}

By Proposition \ref{p5.2}, Corollary \ref{cor3.17} and (\ref{3.8}), it is easy to see
$(0,\infty)\subset \textmd{Dom}(g(y,\cdot))$ and $q(y, r)= q(y, 1)$
hold for $\tilde{\nu}$-a.e. $y\in \mathcal{S}$ and $r> 0$.
If we take $d\nu'(y)=q(y, 1)d\tilde{\nu}(y)$, and endow a measure $\mu:=\nu'\otimes\mathcal{L}^{1}$ on $\mathcal{S}\times(0,\infty)$, then by Theorem \ref{thm3.15}, we have $g_{*}\mu=m$.

Recall the flow $F_{t}$ in Definition \ref{def-flow1}.
From the discussion above, there is an $m$-negligible set $\mathcal{N}\supset X\setminus\mathcal{T}$ such that for all $t\geq0$, $F_{t}:b^{-1}((0,\infty))\setminus \mathcal{N}\rightarrow b^{-1}((t,\infty))\setminus \mathcal{N}$ is a bijection, whose inverse is denoted by  $F_{-t}:=F_{t}^{-1}$, a map from $b^{-1}((t,\infty))\setminus \mathcal{N}$ to $b^{-1}((0,\infty))\setminus \mathcal{N}$.
We can redefine the value of $F_{t}$ on $\mathcal{N}$ to make them Borel on $b^{-1}((0,\infty))$, similarly for $F_{-t}$, then $F_{-t}\circ F_{t}=Id$ $m$-a.e. on $b^{-1}((0,\infty))$, $F_{t}\circ F_{-t}=Id$ $m$-a.e. on $b^{-1}((t,\infty))$.

Suppose $t\geq0$, for any $\varphi\in C_{c}(b^{-1}((t,\infty)))$, we have
\begin{align}\label{5.3}
&\int_{b^{-1}((t,\infty))}\varphi d(F_{t})_{*}m\\
=&\int_{\mathcal{S}}\int_{0}^{\infty}\varphi(F_{t}\circ g(x',s)) d\nu'(x')\otimes d\mathcal{L}^{1}(s)\nonumber\\
=&\int_{\mathcal{S}}\int_{0}^{\infty}\varphi(g(x',s+t)) d\nu'(x')\otimes d\mathcal{L}^{1}(s)\nonumber\\
=&\int_{\mathcal{S}}\int_{t}^{\infty}\varphi(g(x',s)) d\nu'(x')\otimes d\mathcal{L}^{1}(s)\nonumber\\
=&\int_{b^{-1}((t,\infty))}\varphi dm,\nonumber
\end{align}
thus $F_{t}:(b^{-1}((0,\infty)),m)\rightarrow(b^{-1}((t,\infty)),m)$ is measure-preserving.
Similarly, $F_{-t}$ is also measure-preserving on $b^{-1}((t,\infty))$.

In conclusion, we have the following proposition:

\begin{prop}\label{prop5.3}
Suppose $(X,d,m)$ is a noncompact $\textmd{RCD}(0,N)$ space and satisfies (\ref{str-min-vol}), then the ray map $g$ is measure-preserving when viewed as a map from $(\mathcal{S}\times(0,\infty),\mu=\nu'\otimes\mathcal{L}^{1})$ to $(b^{-1}((0,\infty)),m)$, i.e.
\begin{align}
g_{*}\mu=m.
\end{align}
For any $t\geq0$, both $F_{t}:(b^{-1}((0,\infty)),m)\rightarrow(b^{-1}((t,\infty)),m)$ and $F_{-t}:(b^{-1}((t,\infty)),m)\rightarrow(b^{-1}((0,\infty)),m)$ are measure-preserving.
\end{prop}

In the following, we associate each level set $b^{-1}(r)$ a natural measure.

Consider the Busemann function $b$ as a map from $b^{-1}((0,\infty))$ to $\mathbb{R}^{+}$.
By (\ref{5.2}), this map is measure preserving: $b_{*}(m|_{b^{-1}[r_{1},r_{2}]})=c\mathcal{L}^{1}([r_{1},r_{2}])$ for any $r_{2}\geq r_{1}> 0$, where $c=m_{-1}(\Xi_{0}(K))$.
Now we apply Theorem \ref{disin-1} to $(b^{-1}((0,\infty)),\mathcal{B}(b^{-1}((0,\infty))),m)$ and $(\mathbb{R}^{+},\mathcal{B}(\mathbb{R}^{+}),c\mathcal{L}^{1})$,
then we obtain a unique disintegration $r\mapsto\tilde{m}_{r}$ of $m$ over $c\mathcal{L}^{1}$ strongly consistent with $b$.
Hence
\begin{align}\label{5.1}
\int_{b^{-1}((0,\infty))}\varphi dm=c\int_{0}^{\infty}\int\varphi d\tilde{m}_{r}dr, \quad \forall \varphi\in C_{c}(b^{-1}((0,\infty))).
\end{align}
Note that a priori the existence and uniqueness of the $\tilde{m}_{r}$'s only hold in a.e. $r\in\mathbb{R}^{+}$ sense.
In fact, we have
\begin{lem}\label{lem5.1}
The $\tilde{m}_{r}$ can be chosen to be a weakly continuous family, i.e. for any $\varphi\in C_{c}(b^{-1}(0,\infty))$, the map $r\mapsto I_{\varphi}(r):=\int\varphi d\tilde{m}_{r}$ is continuous.
Furthermore, for every $t\geq0$ $(F_{t})_{*}\tilde{m}_{r}=\tilde{m}_{r+t}$ holds for a.e. $r\in\mathbb{R}^{+}$.
\end{lem}
The proof of Lemma \ref{lem5.1} is adapted from Corollary 3.8 in \cite{GigPhi15}, we provide some details here.

\begin{proof}
Firstly, for every $t\geq0$, by (\ref{5.3}) and (\ref{5.1}), for any $\varphi\in C_{c}(b^{-1}((t,\infty)))$, on one hand we have
\begin{align}
&\int_{b^{-1}((t,\infty))}\varphi d(F_{t})_{*}m
=\int_{b^{-1}((0,\infty))}\varphi\circ F_{t} dm\\
=&c\int_{0}^{\infty}\int\varphi\circ F_{t} d\tilde{m}_{r}dr
=c\int_{0}^{\infty}\int\varphi d(F_{t})_{*}\tilde{m}_{r}dr,\nonumber
\end{align}
on the other hand,
\begin{align}
\int_{b^{-1}((t,\infty))}\varphi d(F_{t})_{*}m
=\int_{b^{-1}((t,\infty))}\varphi dm
=c\int_{t}^{\infty}\int\varphi d\tilde{m}_{r}dr
=c\int_{0}^{\infty}\int\varphi d\tilde{m}_{r+t}dr,
\end{align}
hence for every $t\geq0$ $(F_{t})_{*}\tilde{m}_{r}=\tilde{m}_{r+t}$ holds for a.e. $r\in\mathbb{R}^{+}$.

We claim that for any Lipschitz function $\varphi$ with bounded support $\textmd{supp}\varphi\subset b^{-1}(0,\infty)$, the map $r\mapsto I_{\varphi}(r)$ admits a Lipschitz representative.
Since $C_{c}(b^{-1}(0,\infty))$ has a countable dense subset consists of Lipschitz functions with bounded support, by an easy density argument we are easy to check the conclusion of the lemma follows from this claim.

Now for any fixed Lipschitz function $\varphi$ with bounded support $\textmd{supp}\varphi\subset b^{-1}(0,\infty)$, and for a.e. $r\in\mathbb{R}^{+}$,
\begin{align}\label{5.4}
&\bigl|I_{\varphi}(r+t)-I_{\varphi}(r)\bigr|=\bigl|\int\varphi d\tilde{m}_{r+t}-\int\varphi d\tilde{m}_{r}\bigr|\\
=&\bigl|\int\bigl(\varphi\circ F_{t}-\varphi\bigr) d\tilde{m}_{r}\bigr|\leq t\textmd{Lip}(\varphi), \nonumber
\end{align}
where we use (\ref{3.7}) in the last inequality.
By (\ref{5.4}), we are easy to check that the distributional derivative of $r\mapsto I_{\varphi}(r)$ is bounded by $\textmd{Lip}(\varphi)$, thus $r\mapsto I_{\varphi}(r)$ admits a Lipschitz representative.
\end{proof}

In the remaining part of this paper, we will fix some $r'>0$ and denote by $Z=b^{-1}(r')$.

We first prove that under the assumption of Theorem \ref{thm1.2}, $(X,d,m)$ has exactly one end.
We claim that the Busemann function $b$ obtain a global minimum on $X$.
If the claim holds, then by Corollary \ref{cor3.5} we are easy to see $b$ is a proper function on $X$, and then by Lemma \ref{lem3.2} $(X,d,m)$ has only one end.
Suppose on the contrary the claim does not hold, then there is a sequence of points $\{x_{i}\}\subset X$ such that $b(x_{i})=-i$.
From each $x_{i}$, there is a geodesic ray $\eta^{i}:[-i,\infty)\rightarrow X$ such that $\eta^{i}(-i)=x_{i}$, $\eta^{i}(r')\in Z$, and $(x_{i},\eta^{i}(t))\in\Gamma$, $\forall t\geq-i$.
Recall that $Z$ is a compact set.
Hence up to a subsequence, $\eta^{i}$ converge to a line $\eta:(-\infty,\infty)\rightarrow X$ with $\eta(r')\in Z$.
Now by Gigli's splitting theorem on $\textmd{RCD}(0,N)$ (see Theorem \ref{splitting}), $(X,d,m)$ is isomorphic to the product of $(\mathbb{R},d_{Eucl},\mathcal{L}^{1})$ and some $(X',d',m')$, where $(X',d',m')$ is either a point (when $N\in[1, 2)$) or a $\textmd{RCD}(0,N-1)$ space (when $N\geq 2$).
In the case of $N\geq 2$, since $(X',d',m')$ is an $\textmd{MCP}(0,N-1)$ space, if it is noncompact, then it will have infinite volume, thus $(X,d,m)$ cannot have linear volume growth and we get a contradiction.
Hence $X'$ must be a compact space.
It is easy to see the Busemann function $\tilde{b}$ associated to $\eta$ must coincide with $b$ and
\begin{align}
\lim_{R\rightarrow\infty}\frac{m(B_{p}(R))}{R}=2m_{-1}(\Xi_{0}(K)),
\end{align}
which contradicts (\ref{str-min-vol}).
Thus the claim holds.

The remaining arguments in Theorem \ref{thm1.2} is to prove a `volume cone implies metric cone'-type property on $(X,d,m)$.
We follow the strategy in \cite{GigPhi15}.

\subsection{Basic properties of b}

In this subsection, we obtain some basic properties of $b$.

From the argument at the end of last subsection, we have obtained
\begin{lem}
$b$ is a proper function on $X$.
\end{lem}

By Lemma \ref{lem3.1}, it is easy to see $\textmd{lip}(b)\equiv 1$ on $X$.
Since volume doubling property and a weak local $(1,1)$-Poincar\'{e} inequality hold on the $\textmd{RCD}(0,N)$ space $(X,d,m)$ (see \cite{LV07} \cite{Ra12} \cite{Ra12-2}), we have
\begin{lem}\label{l5.6}
$|Db(x)|=\textmd{lip}(b)(x)=1$ $m$-a.e..
\end{lem}

We recall the follow proposition, which is a consequence of the Laplacian comparison estimates for the distance function (\cite{Gig15}).

\begin{prop}[Proposition 5.19 in \cite{Gig15}]\label{subharmonic}
Let $(X,d,m)$ be an $\textmd{RCD}(0,N)$ space.
Let $\gamma$ be a geodesic ray and $b$ the Busemann function associated to it.
Then $b\in D(\mathbf{\Delta})$ and $\mathbf{\Delta}b\geq0$.
\end{prop}

Combining Propositions \ref{prop5.3}, \ref{subharmonic} and Lemma \ref{l5.6}, we can prove the following:

\begin{prop}
$\mathbf{\Delta} b =0$ on $b^{-1}((0,\infty))$.
\end{prop}
\begin{proof}
Let $\varphi: \mathbb{R}^{+}\rightarrow[0,1]$ be a Lipschitz function with $\textmd{supp}(\varphi)\subset\subset(0,\infty)$.
Thus
\begin{align}
0 &=\int_{\mathcal{S}}\biggl(\int_{0}^{\infty}\varphi'(s)ds\biggr)\nu'(dy) \nonumber\\ &=\int_{b^{-1}((0,\infty))}\varphi'(b(x)) dm \nonumber\\
&=\int_{b^{-1}((0,\infty))}\varphi'(b(x))|Db(x)|^{2} dm \nonumber\\
&=-\int_{b^{-1}((0,\infty))}\varphi(b(x))d\mathbf{\Delta}b\leq0. \nonumber
\end{align}
Hence $\int_{b^{-1}((0,\infty))}\varphi(b(x))d\mathbf{\Delta}b=0$.
By $\mathbf{\Delta}b\geq0$ and the arbitrariness of $\varphi$, we obtain $\mathbf{\Delta} b =0$ on $b^{-1}((0,\infty))$.
\end{proof}

Evidently, $b$ is not in $W^{1,2}(X)$ but only in $W^{1,2}_{loc}(X)$.
Given any $\tilde{R}>\tilde{r}>0$, let $\tilde{\varphi}\in C^{\infty}(\mathbb{R})$ be a smooth function satisfying $\tilde{\varphi}\equiv0$ on $(-\infty,\frac{\tilde{r}}{2}]\cup[2\tilde{R},\infty)$, $\tilde{\varphi}(x)=x$ on $[\bar{r},\tilde{R}]$.
Define $\tilde{b}:X\rightarrow\mathbb{R}$ to be
$$\tilde{b}(x)=\tilde{\varphi}\circ b(x).$$

\begin{prop}
$\tilde{b}\in \textmd{Test}(X)$ and $\Delta\tilde{b}\in W^{1,2}(X)\cap L^{\infty}(X)$.
\end{prop}
\begin{proof}
Evidently $\tilde{b}$ is Lipschitz with $\textmd{supp}(\tilde{b})\subset b^{-1}([\frac{\tilde{r}}{2},2\tilde{R}])$.
Recall the facts that $|Db|^{2}=1$ $m$-a.e. and $\mathbf{\Delta} b =0$ on $b^{-1}((0,\infty))$, then by the chain rule for the distributional Laplacian (see Proposition 4.11 in \cite{Gig15}), we have
$$\mathbf{\Delta}\tilde{b}=\tilde{\varphi}''\circ b|Db|^{2}m+ \tilde{\varphi}'\circ b \mathbf{\Delta} b =\tilde{\varphi}''\circ b m.$$
Obviously $\tilde{\varphi}''\circ b\in W^{1,2}(X)\cap L^{\infty}(X)$, hence $\tilde{b}\in D(\Delta)$ with $\Delta\tilde{b}\in W^{1,2}(X)\cap L^{\infty}(X)$.
\end{proof}

\begin{prop} [Euler equation for $b$]\label{prop5.11}
Let $f,g\in Test(X)$ with $\textmd{supp}(f)\subset\subset b^{-1}((0,\infty))$.
Then
\begin{align}\label{5.8}
\int\Delta f\langle\nabla b,\nabla g\rangle dm=\int f\langle\nabla b,\nabla \Delta g\rangle dm.
\end{align}
\end{prop}

\begin{proof}
Evidently to conclude we only need to prove the case $f\geq0$.
Choose $\tilde{R}>\tilde{r}>0$ such that $\textmd{supp}(f)\subset b^{-1}([\tilde{r},\tilde{R}])$, and construct $\tilde{b}$ as in the previous paragraphs.
Let $\epsilon\in\mathbb{R}$, since $f, \tilde{b}+\epsilon g\in \textmd{Test}(X)$ and $f\geq0$, we can apply the Bochner inequality (\ref{2.14}):
\begin{align}\label{5.23}
\frac{1}{2}\int\Delta f|D(\tilde{b}+\epsilon g)|^{2}dm\geq\int f\bigl[\frac{(\Delta (\tilde{b}+\epsilon g))^{2}}{N}+\langle\nabla (\tilde{b}+\epsilon g),\nabla\Delta (\tilde{b}+\epsilon g)\rangle\bigr]dm.
\end{align}
Using the facts that $|D\tilde{b}|^{2}=1$ and $\Delta \tilde{b} =0$ $m$-a.e. on $b^{-1}([\tilde{r},\tilde{R}])$, it is easy to see the equality in (\ref{5.23}) holds when $\epsilon=0$.
Expand (\ref{5.23}), we obtain
$$
\int\Delta f\bigl(\epsilon\langle\nabla g,\nabla \tilde{b}\rangle+ \frac{\epsilon^{2}}{2}|Dg|^{2}\bigr)dm\geq\int f\biggl(\epsilon\langle\nabla \tilde{b},\nabla \Delta g\rangle+ \epsilon^{2}\bigl(\frac{(\Delta g)^{2}}{N}+\langle\nabla g,\nabla\Delta g\rangle\bigr)\biggr)dm.
$$
Divide by $\epsilon>0$ (resp. $\epsilon<0$) and let $\epsilon\downarrow0$ (resp. $\epsilon\uparrow0$) we obtain (\ref{5.8}).
\end{proof}

Using the Euler equation (\ref{5.8}), we obtain the information on $\textmd{Hess}(b)$, which was indicated formally in Remark 4.15 of \cite{Gig13}.

\begin{lem}\label{lem5.10}
$\textmd{Hess}(b)=0$ $m$-a.e. on $b^{-1}((0,\infty))$.
\end{lem}

\begin{proof}

Let $g, h\in \textmd{Test}(X)$ with compact support such that $\textmd{supp}(g),\textmd{supp}(h)\subset b^{-1}([\tilde{r},\tilde{R}])$ for some $\tilde{R}>\tilde{r}>0$.
Construct $\tilde{b}$ as in the previous paragraphs.
Using the Euler equation (\ref{5.8}), the Leibniz rule for the divergence, the facts that $|D\tilde{b}|^{2}=1$ and $\Delta \tilde{b} =0$ $m$-a.e. on $b^{-1}([\tilde{r},\tilde{R}])$, as well as some integration by parts we obtain:
\begin{align}
&\int \textmd{div}(h\nabla g)\langle \nabla \tilde{b},\nabla g\rangle dm \nonumber\\
=&\int\Delta(h g)\langle \nabla \tilde{b},\nabla g\rangle dm-\int g\Delta h\langle \nabla \tilde{b},\nabla g\rangle dm-\int\langle \nabla h,\nabla g\rangle\langle \nabla \tilde{b},\nabla g\rangle dm\nonumber\\
=&\int gh\langle \nabla \tilde{b},\nabla \Delta g\rangle dm-\int\Delta h\langle \nabla \tilde{b},\nabla \bigl( \frac{g^{2}}{2}\bigr)\rangle dm-\int\langle \nabla h,\nabla g\rangle \langle \nabla \tilde{b},\nabla g\rangle dm\nonumber\\
=&\int \langle \nabla \tilde{b},\nabla (gh\Delta g)\rangle dm-\int \Delta g\langle \nabla \tilde{b},\nabla (gh)\rangle dm-\int\langle \nabla h,\nabla g\rangle\langle \nabla \tilde{b},\nabla g\rangle dm-\int\Delta h\langle \nabla \tilde{b},\nabla \bigl( \frac{g^{2}}{2}\bigr)\rangle dm\nonumber\\
=&-\int g\Delta g\langle \nabla \tilde{b},\nabla h)\rangle dm-\int h\Delta g\langle \nabla \tilde{b},\nabla g\rangle dm-\int\langle \nabla h,\nabla g\rangle\langle \nabla \tilde{b},\nabla g\rangle dm-\int\Delta h\langle \nabla \tilde{b},\nabla \bigl( \frac{g^{2}}{2}\bigr)\rangle dm\nonumber\\
=&-\int g\Delta g\langle \nabla \tilde{b},\nabla h)\rangle dm-\int \textmd{div}(h\nabla g)\langle \nabla \tilde{b},\nabla g\rangle dm-\int\Delta h\langle \nabla \tilde{b},\nabla \bigl( \frac{g^{2}}{2}\bigr)\rangle dm.\nonumber
\end{align}
Hence,
\begin{align}
\int \textmd{div}(h\nabla g)\langle \nabla \tilde{b},\nabla g\rangle dm =-\frac{1}{2}\int g\Delta g\langle \nabla \tilde{b},\nabla h)\rangle dm-\frac{1}{2}\int\Delta h\langle \nabla \tilde{b},\nabla \bigl( \frac{g^{2}}{2}\bigr)\rangle dm.
\end{align}
On the other hand,
\begin{align}
&\int h\langle\nabla \tilde{b},\nabla(\frac{|\nabla g|^{2}}{2})\rangle dm\nonumber\\
=&\int \langle\nabla \tilde{b},\nabla\bigl(h\frac{|\nabla g|^{2}}{2}\bigr) dm- \int\frac{|\nabla g|^{2}}{2} \langle\nabla \tilde{b},\nabla h\rangle dm\nonumber\\
=&- \int \Delta\bigl(\frac{g^{2}}{2}\bigr)\langle\nabla \tilde{b},\nabla h\rangle dm + \int \frac{|\nabla g|^{2}}{2}\langle\nabla \tilde{b},\nabla h\rangle dm + \int g\Delta g \langle\nabla \tilde{b},\nabla h\rangle dm \nonumber\\
=&- \int \frac{g^{2}}{2} \langle\nabla \tilde{b},\nabla\Delta h\rangle dm + \int \frac{|\nabla g|^{2}}{2}\langle\nabla \tilde{b},\nabla h\rangle dm + \int g\Delta g \langle\nabla \tilde{b},\nabla h\rangle dm\nonumber\\
=& \int\Delta h\langle\nabla \tilde{b},\nabla\bigl(\frac{g^{2}}{2}\bigr)\rangle dm - \int h\langle\nabla \tilde{b},\nabla(\frac{|\nabla g|^{2}}{2})\rangle dm + \int g\Delta g \langle\nabla \tilde{b},\nabla h\rangle dm.\nonumber
\end{align}
Hence
\begin{align}
\int h\langle\nabla \tilde{b},\nabla(\frac{|\nabla g|^{2}}{2})\rangle dm=
\frac{1}{2}\int\Delta h\langle\nabla \tilde{b},\nabla\bigl(\frac{g^{2}}{2}\bigr)\rangle dm + \frac{1}{2} \int g\Delta g \langle\nabla \tilde{b},\nabla h\rangle dm.
\end{align}
By (\ref{def-hess}) and polarization, we obtain $\textmd{Hess}(b)=0$ $m$-a.e. on $b^{-1}((0,\infty))$.
\end{proof}

\subsection{Effect of $F_{t}$ on the Dirichlet energy and metric}

Because we only know information on $b^{-1}((0,\infty))$, we need some cut-off argument.
Let $\bar{r}$ be a fixed positive number.
Let $\bar{\varphi}\in C^{\infty}(\mathbb{R})$ be a smooth increasing function satisfying $\bar{\varphi}\equiv0$ on $(-\infty,\frac{\bar{r}}{2}]$, $\bar{\varphi}(x)=x$ on $[\bar{r},\infty)$.
Define $\bar{b}:X\rightarrow\mathbb{R}$ to be
$$\bar{b}(x)=\bar{\varphi}\circ b(x).$$
Note that $\bar{b}$ is Lipschitz, with support in $b^{-1}([\frac{\bar{r}}{2},\infty))$ and equal to $b$ on $b^{-1}([\bar{r},\infty))$.
We define the reparametrization function $\textmd{rep}:\mathbb{R}^{+}\times\mathbb{R}\rightarrow \mathbb{R}^{+}$, $\textmd{rep}_{t}(r)=\textmd{rep}(t,r)$ by requiring it satisfying $\textmd{rep}_{0}(r)=0$ and
\begin{align}\label{5.10}
\partial_{t}\textmd{rep}_{t}(r)=\bar{\varphi}'(r+\textmd{rep}_{t}(r))
\end{align}
for every $t\geq0$.
It's easy to see $\textmd{rep}_{t}(r)=0$ for $(t,r)\in\mathbb{R}^{+}\times(-\infty,\frac{\bar{r}}{2}]$ and $\textmd{rep}_{t}(r)=t$ on $(t,r)\in\mathbb{R}^{+}\times[\bar{r},\infty)$.
We then define the flow $\bar{F}:\mathbb{R^{+}}\times (X\setminus\mathcal{N})\rightarrow X\setminus\mathcal{N}$ to be
$$
\bar{F}_{t}(x)=F_{\textmd{rep}_{t}(b(x))}(x).
$$
$\bar{F}_{t}: X\setminus\mathcal{N}\rightarrow X\setminus\mathcal{N}$ is invertible for every $t\geq0$.
Thus for $t > 0$ the map $\bar{F}_{-t}=(\bar{F}_{t})^{-1}: X\setminus\mathcal{N}\rightarrow X\setminus\mathcal{N}$ is well defined.
We also suitably define $\bar{F}_{t}$ and $\bar{F}_{-t}$ on $\mathcal{N}$ to make them Borel maps on $X$.

\begin{lem}\label{prop5.9}
The following properties hold:
\begin{description}
  \item[(i)] $\bar{F}_{t}$ is the identity on $b^{-1}((-\infty,\frac{\bar{r}}{2}])\setminus\mathcal{N}$, $\bar{F}_{t}$ sends $b^{-1}((0,\infty))\setminus\mathcal{N}$ into itself for every $t\geq0$.
  \item[(ii)] $\bar{F}_{t}$ coincides with $F_{t}$ in $b^{-1}([\bar{r},\infty))\setminus\mathcal{N}$.
  \item[(iii)] For every $x\in X\setminus\mathcal{N}$ the curve $[0,\infty)\ni t\mapsto\eta^{(x)}_{t}:=\bar{F}_{t}(x)$ satisfies
  \begin{align}\label{5.11}
  \bar{b}(\eta^{(x)}_{t})=\bar{b}(\eta^{(x)}_{s})+\frac{1}{2}\int_{s}^{t}\bigl[|\dot{\eta}^{(x)}_{r}|^{2}+\textmd{lip}(\bar{b})^{2} (\eta^{(x)}_{r})\bigr]dr,
  \qquad \forall 0\leq s\leq t.
  \end{align}
  In particular, the speed of $\eta^{(x)}_{t}$ is equal to $\textmd{lip}(\bar{b})(\bar{F}_{t}(x))$ for every $t$, thus granting that $t\rightarrow \bar{F}_{t}(x)$ is $Lip(\bar{b})$-Lipschitz for every $x\in X\setminus\mathcal{N}$.
  \item[(iv)] For every $t\geq0$,
  \begin{align}\label{5.9}
  c(t)m\leq(\bar{F}_{t})_{*}m\leq C(t)m,
  \end{align}
  where $c,C:\mathbb{R}^{+}\rightarrow(0,\infty)$ are continuous functions.
  \item[(v)] The maps $\bar{F}_{t}$ form a group, i.e. $\bar{F}_{0}= \textmd{Id}$ $m$-a.e. and
  $$\bar{F}_{t+s}=\bar{F}_{t}\circ \bar{F}_{s}\quad m\text{-a.e.}$$
  for every $t,s\in \mathbb{R}$.
\end{description}
\end{lem}

\begin{proof}
Properties (i), (ii), (v) follows from the definition and the obvious properties of the reparametrization function $\textmd{rep}$ and the flow $F_{t}$.

Now we prove (iii).
Note that for $x\in X\setminus\mathcal{N}$,
$$b(\eta^{(x)}_{t})=b(\bar{F}_{t}(x))=b(F_{\textmd{rep}_{t}(b(x))}(x))=b(x)+\textmd{rep}_{t}(b(x)).$$
Recall that $\bar{b}=\bar{\varphi}\circ b$, we have $$\textmd{lip}(\bar{b})(\eta^{(x)}_{t})=|\bar{\varphi}'|(b(\eta^{(x)}_{t}))\textmd{lip}(b)(\eta^{(x)}_{t})= |\bar{\varphi}'(b(x)+\textmd{rep}_{t}(b(x)))|.$$
On the other hand, $|\dot{\eta}^{(x)}_{t}|=|\dot{\sigma}^{(x)}_{s}||\partial_{t}\textmd{rep}_{t}(b(x))|=|\partial_{t}\textmd{rep}_{t}(b(x))|$, where $[0,\infty)\ni s\mapsto\sigma^{(x)}_{s}:=F_{s}(x)$ is the geodesic ray emitting from $x$.
Hence from (\ref{5.10}), we have $|\dot{\eta}^{(x)}_{t}|=\textmd{lip}(\bar{b})(\eta^{(x)}_{t})$.
Note that $$\frac{d}{dt}\bar{b}(\eta^{(x)}_{t})=\bar{\varphi}'(b(\eta^{(x)}_{t}))\frac{d}{dt}b(\eta^{(x)}_{t}) =\bar{\varphi}'(b(\eta^{(x)}_{t}))\frac{\partial}{\partial t}\textmd{rep}_{t}(b(x))=\frac{1}{2}[|\dot{\eta}^{(x)}_{t}|^{2}+\textmd{lip}(\bar{b})^{2}(\eta^{(x)}_{t})],$$
integrate over $[s,t]\subset[0,\infty)$ and we obtain (\ref{5.11}).

Now we prove (iv).
Since $\bar{F}_{t}$ is identity on $b^{-1}((-\infty,\frac{\bar{r}}{2}])\setminus\mathcal{N}$, we have $$(\bar{F}_{t})_{*}m|_{(-\infty,\frac{\bar{r}}{2}])}=m|_{(-\infty,\frac{\bar{r}}{2}])},\qquad \forall t\geq0.$$
In the following we consider the behavior of $\bar{F}_{t}$ on $b^{-1}((0,\infty))$.

For $t\geq0$, denote by $f_{t}(s):=s+\textmd{rep}_{t}(s)$.
By definition, we have
$$\bar{F}_{t}(y)=g(x',s+\textmd{rep}_{t}(s))=g(x',f_{t}(s))$$
for any $y=g(x',s)\in b^{-1}((0,\infty))$.

Let $\varphi$ be any Borel function with $\textmd{supp}(\varphi)\subset b^{-1}((0,\infty))$, then
\begin{align}
&\int_{b^{-1}((0,\infty))}\varphi d(\bar{F}_{t})_{*}m\nonumber\\
=&\int_{\mathcal{S}}\int_{0}^{\infty}\varphi(\bar{F}_{t}\circ g(x',s)) d\nu'(x')\otimes d\mathcal{L}^{1}(s)\nonumber\\
=&\int_{\mathcal{S}}\int_{0}^{\infty}\varphi(g(x',f_{t}(s))) d\nu'(x')\otimes d\mathcal{L}^{1}(s)\nonumber\\
=&\int_{\mathcal{S}}\int_{0}^{\infty}\varphi(g(x',s))(\partial_{s}f^{-1}_{t}(s)) d\nu'(x')\otimes d\mathcal{L}^{1}(s).\nonumber
\end{align}

Thus to complete the proof of (iv) we only need to prove that there are functions $c(t),C(t):\mathbb{R}^{+}\rightarrow \mathbb{R}^{+}$ such that $\forall r\in\mathbb{R}^{+}$ we have
\begin{align}\label{5.12}
c(t)\leq\partial_{r}f^{-1}_{t}(r)\leq C(t).
\end{align}

Denote by $h_{t}(r):=\partial_{r}\textmd{rep}_{t}(r)=\partial_{r}f_{t}(r)-1$.
Differentiate (\ref{5.10}) in $r$ to deduce that $h_{t}(r)$ solves the equation
\begin{align}\label{5.13}
\partial_{t}h_{t}(r)=\bar{\varphi}''(r+\textmd{rep}_{t}(r))(1+h_{t}(r))
\end{align}
with the initial condition $h_{0}(r)=0$.
Note that the function identically $-1$ is a solution of (\ref{5.13}), also note that $\bar{\varphi}''$ has compact support, by comparison principle we are easy to derive
$$h_{t}(r)\geq-1+\tilde{c}(t), \qquad \forall r\geq0,$$
for a function $\tilde{c}(t)>0$.
On the other hand, since $\textmd{rep}_{t}(r)=t$ for $r\geq\bar{r}$, we know $h_{t}(r)=0$ for all $r\geq\bar{r}$.
Also note that $h_{t}(0)=0$.
Then by the smoothness of $h_{t}(r)$, we have $h_{t}(r)\leq \tilde{C}(t)$ for any $r\in\mathbb{R}^{+}$.
Hence
$$
\frac{1}{1+\tilde{C}(t)}\leq\partial_{r}f^{-1}_{t}(r)\leq \frac{1}{\tilde{c}(t)}
$$
for any $r\in\mathbb{R}^{+}$.
This completes the proof of (iv).
\end{proof}

The properties in Lemma \ref{prop5.9} are similar to the properties a)- f) in Proposition 3.9 of \cite{GigPhi15}.
Using these properties we can establish the following lemma, which is very similar to Lemma 3.11 in \cite{GigPhi15}.

\begin{lem}\label{lem5.11}
If $f\in L^{p}(X)$, $p <\infty$, then $t\mapsto f\circ \bar{F}_{t}\in L^{p}(X)$ is continuous.

If $f\in W^{1,2}(X)$, then $t\mapsto f\circ \bar{F}_{t}\in L^{2}(X)$ is $C^{1}$ and its derivative is given by
\begin{align}\label{5.16}
\frac{d}{dt}f\circ \bar{F}_{t} = \langle \nabla f,\nabla \bar{b} \rangle\circ \bar{F}_{t}.
\end{align}
\end{lem}

We omit the detailed proof of Lemma \ref{lem5.11}, the reader can refer to \cite{GigPhi15} for the very same argument.

Our next goal is to study the effect of $\bar{F}_{t}$ on the Dirichlet energy $\mathcal{E}$, see Theorem \ref{thm5.15}.
The main tool is the heat flow, which is the $L^{2}$-gradient flow of $\mathcal{E}$.
A problem arises in the proof of Theorem \ref{thm5.15} is that the support of $h_{s}(f)$ will not stay in $b^{-1}((\bar{r},\infty))$ even if the support of $f$ does.
Thus we need suitable cut-off functions and controlling the error terms.
Our argument follows \cite{GigPhi15} closely.

Firstly we recall the following lemma from \cite{GigPhi15}:

\begin{lem}[Lemma 3.14 in \cite{GigPhi15}]\label{lem5.12}
Suppose $(X,d,m)$ is an $\textmd{RCD}(0,N)$ space.
For every $r>0$ there is a constant $C(r)>0$ such that the following holds.
Given $K\subset U$ with $K$ compact and $U$ open such that $\inf_{x\in K,y\in U^{c}}d(x,y)\geq r$.
If $f\in L^{2}(X)$ with $\textmd{supp}(f)\subset K$, then for every $t>0$ the quantities
\begin{align}\label{5.25}
&\int_{U^{c}}|h_{t}(f)|^{2}dm,\quad \int_{0}^{t}\int_{U^{c}}|h_{s}(f)|^{2}dmds, \quad \int_{0}^{t}\int_{U^{c}}|Dh_{s}(f)|^{2}dmds,\\
& \int_{0}^{t}\int_{U^{c}}|\Delta h_{s}(f)|^{2}dmds, \quad
\int_{0}^{t}\int_{U^{c}}|D\Delta h_{s}(f)|^{2}dmds\nonumber
\end{align}
are all bounded from above by $C(r)t^{2}\parallel f\parallel^{2}_{L^{2}}$.
\end{lem}

Applying Lemma \ref{lem5.12}, we can prove the following estimate.

\begin{prop}\label{prop5.15}
Given any $r>0$ there is a constant $C(r)$ such that
for any $f\in L^{2}(X)$ with $\textmd{supp}(f)\subset b^{-1}((\bar{r}+r,\infty))$ and any $s\in(0,1)$ we have
\begin{align}\label{5.15}
\biggl|\int\langle \nabla h_{2s}(f),\nabla\bar{b} \rangle f dm\biggr|\leq  s^{2} C(r)\parallel f \parallel^{2}_{L^{2}}.
\end{align}
\end{prop}

\begin{proof}
By (\ref{5.24-3}), it is easy to check that the term $\int\langle \nabla h_{2s}(f),\nabla\bar{b} \rangle f dm$ vary continuously as $f$ varies in $L^{2}(X)$.
Hence by Lemma \ref{lem2.13} and a simple density argument we only need to prove the conclusion of the proposition holds for $f\in \textmd{Test}(X)$ with support contained in a compact set $K$ with $K\subset b^{-1}((\bar{r}+r,R))$.
Denote by $U=b^{-1}((\bar{r}+\frac{r}{2},2R))$ and note that $\inf_{x\in K,y\in U^{c}}d(x,y)\geq \frac{r}{2}$.

Since the function $t\mapsto \int \langle \nabla h_{s+t}(f),\nabla\bar{b}\rangle h_{s-t}(f) dm$ is absolutely continuous on $[0,s]$, we have
\begin{align}\label{5.22}
\int \langle \nabla h_{2s}(f),\nabla\bar{b}\rangle f dm=\int \langle \nabla h_{s}(f),\nabla\bar{b}\rangle h_{s}(f) dm + \int_{0}^{s}\frac{d}{dt}\int \langle \nabla h_{s+t}(f),\nabla\bar{b}\rangle h_{s-t}(f) dm dt.
\end{align}
Note that
\begin{align}\label{5.28}
&\int \langle \nabla h_{s}(f),\nabla\bar{b}\rangle h_{s}(f) dm= \int \langle \nabla \frac{|h_{s}(f)|^{2}}{2},\nabla\bar{b}\rangle dm \\
\leq & \parallel\Delta \bar{b}\parallel_{L^{\infty}} \int_{U^{c}}  \frac{|h_{s}(f)|^{2}}{2} dm
\leq C(r)s^{2}\parallel\Delta \bar{b}\parallel_{L^{\infty}}\parallel f\parallel^{2}_{L^{2}}, \nonumber
\end{align}
where we use Lemma \ref{lem5.12} in the last inequality.

Let $\chi$ be the cut-off function given by Lemma \ref{lem5.13} relative to the compact set $b^{-1}([\bar{r}+\frac{r}{2},2R])$ and the open set  $b^{-1}((\bar{r},4R))$, then

\begin{align}\label{5.26}
&\int_{0}^{s}\frac{d}{dt}\int \langle \nabla h_{s+t}(f),\nabla\bar{b}\rangle h_{s-t}(f) dm dt \\
=&\int_{0}^{s}\int \bigl(\langle \nabla \Delta h_{s+t}(f),\nabla\bar{b}\rangle h_{s-t}(f)-\langle \nabla h_{s+t}(f),\nabla\bar{b}\rangle \Delta h_{s-t}(f)\bigr) dm dt \nonumber\\
=&\int_{0}^{s}\int \bigl(\langle \nabla \Delta h_{s+t}(f),\nabla\bar{b}\rangle (\chi h_{s-t}(f))-\langle \nabla h_{s+t}(f),\nabla\bar{b}\rangle \Delta (\chi h_{s-t}(f))\bigr) dm dt \nonumber\\
&+ \int_{0}^{s}\int \bigl(\langle \nabla \Delta h_{s+t}(f),\nabla\bar{b}\rangle (1-\chi)h_{s-t}(f)-\langle \nabla h_{s+t}(f),\nabla\bar{b}\rangle \Delta ((1-\chi) h_{s-t}(f))\bigr) dm dt \nonumber\\
\leq& S^{2} \int_{0}^{s}\int_{U^{c}}\bigl(|h_{s-t}(f)||D\Delta h_{s+t}(f)|+|D h_{s+t}(f)|[|h_{s-t}(f)|+|D h_{s-t}(f)|+|\Delta h_{s-t}(f)|]\bigr)dmdt \nonumber\\
\leq& C(r)s^{2}\parallel f\parallel^{2}_{L^{2}},  \nonumber
\end{align}
where $S:=\max\{1, \textmd{Lip}(\bar{b}), \textmd{Lip}(χ\chi), \parallel\Delta\chi\parallel_{L^{\infty}}\}$, with $S\leq C(r)$ by Lemma \ref{lem5.13}.
In (\ref{5.26}) we use the fact that $\textmd{supp}(\chi h_{s-t}(f))\subset\subset b^{-1}((\bar{r},\infty))$ as well as Proposition \ref{prop5.11} to conclude that the first integral in the third equality vanishes, and use Lemma \ref{lem5.12} in the last inequality.

Combining (\ref{5.22}) (\ref{5.28}) and (\ref{5.26}) we obtain the estimate (\ref{5.15}).
\end{proof}

\begin{thm}\label{thm5.15}
Suppose $f\in L^{2}(X)$ satisfies $\textmd{supp}(f)\subset b^{-1}((\bar{r}+T,\infty))$ for some $T\geq0$. Then
\begin{align}\label{5.27}
\mathcal{E}(f\circ \bar{F}_{t})=\mathcal{E}(f), \quad \forall t\leq T.
\end{align}
In particular, $f\in W^{1,2}(X)$ if and only if $f\circ \bar{F}_{t}\in W^{1,2}(X)$.
\end{thm}

\begin{proof}
Let $f_{t}:=f\circ \bar{F}_{t}$ and notice that $\textmd{supp}(f_{t})\subset b^{-1}((\bar{r}+T-t,\infty))\subset b^{-1}((\bar{r},\infty))$ for every $t\leq T$.
By Proposition \ref{prop5.3} we have
$\int|f_{t}|^{2} dm=\int|f|^{2} dm$ and hence
$\frac{d}{dt}\int|f_{t}|^{2} dm=0$
for any $t\leq T$.

Recall that by Lemma \ref{lem5.11}, $t\mapsto f_{t} \in L^{2}(X)$ is continuous.
By (\ref{5.24-2}) and (\ref{5.24-3}), we have
$$\parallel h_{s}(f_{t_{1}})-h_{s}(f_{t_{0}})\parallel_{W^{1,2}}\leq \frac{2}{\sqrt{2s}}\parallel f_{t_{1}}-f_{t_{0}}\parallel_{L^{2}},$$
for any $s\in(0,\frac{1}{2})$.
Thus by Lemma \ref{lem5.11} again, for any fixed $s\in(0,\frac{1}{2})$, $t\rightarrow h_{s}(f_{t}) \in L^{2}(X)$ is $C^{1}$.
Hence for $t\leq T$ we have
\begin{align}\label{5.34}
\frac{1}{2}\frac{d}{dt}\int|h_{s}(f_{t})|^{2}dm &=\lim_{h\rightarrow0}\int h_{s}(f_{t})\frac{h_{s}(f_{t}\circ \bar{F}_{h})-h_{s}(f_{t})}{h} dm\\
&=\lim_{h\rightarrow0}\int h_{2s}(f_{t})\frac{f_{t}\circ \bar{F}_{h}-f_{t}}{h} dm \nonumber\\
&=\lim_{h\rightarrow0}\int f_{t}\frac{h_{2s}(f_{t})\circ \bar{F}_{-h}-h_{2s}(f_{t})}{h}  dm \nonumber\\
&=-\lim_{h\rightarrow0}\int f_{t}\langle \nabla h_{2s}(f_{t}),\nabla \bar{b} \rangle  dm, \nonumber
\end{align}
where we use (\ref{5.16}) in the last equality.

Denote by
$$G(t,s):=\int\frac{|f_{t}|^{2}-|h_{s}(f_{t})|^{2}}{4s}dm.$$
For any $s\in(0,\frac{1}{2})$ the map $t\mapsto G(t, s)\in L^{2}(X)$ is $C^{1}$ with
\begin{align}\label{5.29}
\frac{d}{dt}G(t, s)=\frac{1}{2s}\lim_{h\rightarrow0}\int f_{t}\langle \nabla h_{2s}(f_{t}),\nabla \bar{b} \rangle  dm
\end{align}
for $t\leq T$.
By Proposition \ref{prop5.15} we have
\begin{align}
\biggl|\int\langle \nabla h_{2s}(f_{t}),\nabla\bar{b} \rangle f_{t} dm\biggr|\leq  C s^{2}\parallel f_{t} \parallel^{2}_{L^{2}},
\end{align}
hence
\begin{align}\label{5.32}
|\frac{d}{dt}G(t, s)|\leq C s\parallel f_{t} \parallel^{2}_{L^{2}}\leq C \parallel f \parallel^{2}_{L^{2}}
\end{align}
for $s\in(0,\frac{1}{2})$.

Suppose $f\in W^{1,2}(X)$.
By (\ref{5.24-1}), we have $G(0, s)\leq \mathcal{E}(f)<\infty$.
By (\ref{5.32}), we deduce
\begin{align}\label{5.31}
G(t, s)\leq\mathcal{E}(f)+C|t|\parallel f \parallel^{2}_{L^{2}}
\end{align}
for $t\leq T$ and $s\in(0,\frac{1}{2})$.
Note that $G(t,s)\uparrow\mathcal{E}(f_{t})$ as $s\downarrow0$, by (\ref{5.31}) we have
$\mathcal{E}(f_{t})\leq\mathcal{E}(f)+C|t|\parallel f \parallel^{2}_{L^{2}}<\infty$.
Thus $f_{t}\in W^{1,2}(X)$ for every $t\leq T$.
Passing to the limit as $s\downarrow0$ in (\ref{5.29}), we obtain
$$\frac{d}{dt}\mathcal{E}(f_{t})=0$$
for $t\leq T$.
Hence we obtain (\ref{5.27}).

Suppose $f\notin W^{1,2}(X)$, i.e. $\mathcal{E}(f)=+\infty$.
Then it is impossible that $f_{t}\in W^{1,2}(X)$ for any $t\leq T$.
Otherwise change the role of $f$ and $f_{t}$ in the above argument, we have $\mathcal{E}(f)\leq\mathcal{E}(f_{t})+C|t|\parallel f \parallel^{2}_{L^{2}}<\infty$, which is a contradiction.
Thus we have complete the proof.
\end{proof}

Theorem \ref{thm5.15} can be `localized' as follows:

\begin{cor}\label{cor5.15}
Suppose $f\in L^{2}(X)$ with $\textmd{supp}(f)\subset b^{-1}((T,\infty))$ with $T\geq0$, then $f\in W^{1,2}(X)$ if and only if $f\circ F_{t}\in W^{1,2}(X)$ for any $t\leq T$ and in this case
\begin{align}\label{5.39}
|D(f\circ F_{t} )| = |Df|\circ F_{t}\qquad m\text{-a.e.}.
\end{align}
\end{cor}
\begin{proof}[Sketch of the proof]
Choose $\bar{r}>0$ such that $\textmd{supp}(f)\subset b^{-1}([\bar{r}+T,\infty))$ and then build a corresponding function $\bar{b}$ and the flow $\bar{F}_{t}$ as we have done in this subsection.
Obviously, $f\circ \bar{F}_{t}=f\circ F_{t}$ $m$-a.e., and by the locality property of minimal weak upper gradient, we have $|D(f\circ \bar{F}_{t})|=|D(f\circ F_{t})|$ $m$-a.e..
By Theorem \ref{thm5.15} we deduce that $f\in W^{1,2}(X)$ if and only if $f\circ \bar{F}_{t}\in W^{1,2}(X)$.
Thus we only need to prove $|D(f\circ \bar{F}_{t} )| = |Df|\circ \bar{F}_{t}$ $m$-a.e. for $f\in W^{1,2}(X)$ with $\textmd{supp}(f)\subset b^{-1}((T,\infty))$.
The later can be obtained by the same proof of Corollary 3.17 in \cite{GigPhi15} or Lemma 4.17 in \cite{Gig13}.
\end{proof}

The notion of Sobolev-to-Lipschitz property is a key to deduce metric information from the study of Sobolev functions.
See Proposition 4.20 in \cite{Gig13}.

In Corollary \ref{cor5.15}, (\ref{5.39}) holds only for $f\in W^{1,2}(X)$ with $\textmd{supp}(f)\subset b^{-1}((T,\infty))$, we obtain that $F_{t}$'s are local isometry instead of isometry, see Theorem \ref{thm5.16}.
The arguments here follows that of Theorem 3.18 in \cite{GigPhi15}.

\begin{thm}\label{thm5.16}
If we define the map $F:\mathbb{R}^{+}\times b^{-1}((0,\infty))\rightarrow b^{-1}((0,\infty))$ to be $F(t,x)=F_{t}(x)$, then
$F$ admits a locally Lipschitz representative w.r.t. the measure $\mathcal{L}^{1}\times m$.
We still denote such representative by $F$.
Furthermore, $F$ satisfies the following:
\begin{enumerate}
  \item for every $t\in \mathbb{R}^{+}$, $F_{t}$ is an invertible locally isometry from $b^{-1}((0,\infty))$ to $b^{-1}((t,\infty))$;
  \item for every $t, s \in \mathbb{R}^{+}$ and $x\in b^{-1}((0,\infty))$, we have
  \begin{align}
  &F_{t+s}(x)=F_{t}(F_{s}(x)),\label{5.37}\\
  &d(F_{s}(x),F_{t}(x))=|s-t|;\label{5.38}
  \end{align}

  \item for any curve $\eta:[0,1]\ni s\mapsto b^{-1}((0,\infty))$, putting $\tilde{\eta} := F_{t}\circ\eta$ then $\eta$ is absolutely continuous if and only if $\tilde{\eta}$ is and in this case $|\dot{\tilde{\eta}}_{s}|=|\dot{\eta}_{s}|$ for a.e. $s\in[0, 1]$.
\end{enumerate}
\end{thm}

\begin{proof}[Sketch of the proof]
(2) (3) and the conclusion that $F$ admits a locally Lipschitz representative are easy to prove once we can prove $F_{t}$ has a locally isometric representative for every $t\geq 0$ .
For any $t\in \mathbb{R}^{+}$ and $x_{0}\in b^{-1}((t,\infty))$, choose $r > 0$ such that $B_{x_{0}}(5r)\subset b^{-1}((t,\infty))$.
Let $\Lambda$ be a countable dense subset of the set of 1-Lipschitz functions with support in $B_{x_{0}}(5r)$.
We can prove $d(y_{0},y_{1}) = \sup_{f\in \Lambda} |f(y_{1})-f(y_{0})|, \forall y_{0}, y_{1}\in B_{x_{0}}(r)$, as in \cite{GigPhi15}.
By Corollary \ref{cor5.15}, for any $f\in \Lambda$, $|D(f\circ F_{t})|=|Df|\circ F_{t}\leq1$ $m$-a.e., then by the Sobolev-to-Lipschitz property of $X$, $f\circ F_{t}$ has a $1$-Lipschitz representative.
Since $\Lambda$ is countable, there exists a $m$-negligible Borel set $N$ such that the restriction of $f\circ F_{t}$ to $X\setminus N$ is $1$-Lipschitz for every $f\in \Lambda$.
Hence for any $x_{1}, x_{2}\in F_{t}^{-1}(B_{x_{0}}(r)\setminus N)$, we have $d(F_{t}(x_{1}),F_{t}(x_{2}))=\sup_{f\in\Lambda}|f(F_{t}(x_{1}))-f(F_{t}(x_{2}))|\leq d(x_{1},x_{2}).$
Since $b^{-1}((t,\infty))$ is $\sigma$-compact, by an easy covering argument, $F_{t}$ is locally $1$-Lipschitz on $b^{-1}((0,\infty))\setminus N'$ for some $m$-negligible set $N'$.
The same argument can be applied to $F_{t}^{-1}:b^{-1}((t,\infty))\rightarrow b^{-1}((0,\infty))$.
Hence $F_{t}$ has an invertible locally isometric representative.
\end{proof}

\subsection{Equipped Z with the induced distance and measure}

Now we consider the cross section $Z=b^{-1}(r')$.
Recall that $Z$ is compact by Theorem \ref{compact-level-set}.

Suppose $Z$ consists of exactly one point, i.e. $Z=\{x\}$, it is easy to see in this case $X=R(x)$. 
From Theorem \ref{thm5.16}, it is easy to see  $b^{-1}([0,\infty))$ consists of a geodesic ray emitting from $b^{-1}(0)$.
Then combining with Lemma \ref{lem5.1}, we know $(b^{-1}([0,\infty)),d,m)$ is isomorphic to $(\mathbb{R}^{+},d_{Ecul},c\mathcal{L}^{1})$ with $c=m(b^{-1}([0,1]))$. 

We are going to prove that $b$ is an injective function on $X$. 
If this holds, then combining with the fact that $b$ is a proper function, we know $X$ is isometric to a half-line. 
Suppose $b$ is not injective on $X$, then we can find three distinct points $x_{0},x_{1},x_{2}\in X$ such that $b(x_{1})=b(x_{2})<b(x_{0})\leq 0$, and $b^{-1}([b(x_{0}),\infty))$ is isometric to a half-line. 
It is easy to see $d(x,x_{1})=d(x,x_{1})$ for any $x$ with $b(x)\geq b(x_{0})$. 
Let $\mu=\frac{1}{m(b^{-1}([0,1]))}m|_{b^{-1}([0,1])}$, $\nu=\frac{1}{2}(\delta_{x_{1}}+\delta_{x_{2}})$. 
It is easy to check that the plan $\mu\times \nu$ is optimal, but it is not induced by a map. 
This contradicts to Theorem \ref{3.25}.

In conclusion, we have

\begin{lem}
If $Z$ consists of exactly one point, then $(b^{-1}([0,\infty)),d,m)$ is isomorphic to $(\mathbb{R}^{+},d_{Ecul},c\mathcal{L}^{1})$, where $c=m(b^{-1}([0,1]))$. 
Furthermore, $(X, d)$ is isometric to some $([\bar{r},\infty),d_{Ecul})$ with $\bar{r}\leq 0$.
\end{lem}

From now on, we will always assume $Z$ contains at least 2 points.

We define the projection map $\textmd{Pj}:b^{-1}((0,\infty))\rightarrow Z$ to be
$$\textmd{Pj}(x)=F_{r'-b(x)}(x).$$
By Theorem \ref{thm5.16} the map $\textmd{Pj}$ is well defined and locally Lipschitz.

Suppose $x,y\in Z$, by Lemma \ref{lem3.2}, there is a Lipschitz curve $\eta:[0,1]\ni t\mapsto\eta_{t}\in b^{-1}([r',\infty))$ with $\eta(0)=x'$, $\eta(1)=y'$.
By Theorem \ref{thm5.16}, the curve $\sigma=\textmd{Pj}(\eta)$ is a Lipschitz curve connecting $x'$ and $y'$ in $Z$.
In conclusion, we have

\begin{cor}\label{cor5.17}
If $Z$ consists of more than one point, then for every $x', y' \in Z$ there is a Lipschitz curve $\sigma:[0,1]\rightarrow Z$ with $\sigma_{0}=x'$, $\sigma_{1}=y'$.
In particular, $Z$ is a path-connected space with Hausdorff dimension at least $1$.
\end{cor}

We define a new distance $d'$ on $Z$:
$$d'(x', y'):=\inf_{\sigma}\biggl(\int_{0}^{1}|\dot{\sigma}_{t}|^{2} dt\biggr)^{\frac{1}{2}}$$
for $x', y'\in Z$,
where the infimum is taken among all Lipschitz curves $\sigma: [0, 1]\rightarrow Z \subset X$ and the metric speed is computed w.r.t. the distance $d$.

By Corollary \ref{cor5.17}, $d'$ is always finite and is a distance on $Z$.
It is easy to prove that a curve $\sigma:[0,1]\rightarrow Z\subset X$ is absolutely continuous w.r.t. $d'$ if and only if it is absolutely continuous w.r.t. $d$ and in this case the metric speeds computed w.r.t. the two distances are the same,
so we still use $|\dot{\sigma}_{t}|$ to denote the metric speed if there is no ambiguity.
We can easily check that the map $\textmd{Pj}$ is still locally Lipschitz even if $Z$ equipped with the new distance $d'$.

The measure $m'$ on $Z$ is chosen to be $m':=\tilde{m}_{r'}$. 
By Lemma \ref{lem5.1}, we are easy to check 
\begin{align}\label{5.5}
\textmd{Pj}_{*}(m|_{b^{-1}([c,d])}) = m(b^{-1}([c,d]))m'.
\end{align}
holds for any $d>c>0$. 

Combining the facts that $X$ is doubling, that $\textmd{Pj}$ is locally Lipschitz, and Lemma \ref{lem5.1}, we can derive that $(Z,d',m')$ is also doubling, see Proposition 3.26 in \cite{GigPhi15} for similar proof.

Denote by $Y=Z\times (0,\infty)$, we endow $Y$ with the product measure $m'\otimes \mathcal{L}^{1}$ and the product distance $d'\times d_{Eucl}$ defined by
$$d'\times d_{Eucl}((x',t),(y',s)):=\sqrt{d'(x',y')^{2}+|t-s|^{2}}.$$
We denote $(Y,d'\times d_{Eucl},m'\otimes\mathcal{L}^{1})$ by $(Y,d_{Y},m_{Y})$ for simplicity.

For $0\leq c<d\leq\infty$, we denote by $Y_{(c,d)}\subset Y$ to be $Y_{(c,d)}:=Z\times(c,d)$ for simplicity, similar notations are $Y_{[c,d]}$ e.t.c..

We introduce the maps $T:Y\rightarrow b^{-1}((0,\infty))$ and $S:b^{-1}((0,\infty))\rightarrow Y$ to be
$$T(x',t):=F_{t-r'}(x'),$$
$$S(x):=(\textmd{Pj}(x),b(x)).$$

It's easy to see that $S\circ T=\textmd{Id}|_{Y}$ and $T\circ S=\textmd{Id}|_{b^{-1}((0,\infty))}$.
From Lemma \ref{lem5.1}, it is easy to check
\begin{align}\label{5.43}
T_{*}m_{Y}=m|_{b^{-1}((0,\infty))},\qquad \text{ and }\qquad S_{*}m|_{b^{-1}((0,\infty))}=m_{Y}.
\end{align}

By Theorem \ref{thm5.16}, both $S$ and $T$ are locally Lipschitz.

\subsection{Estimate on the speed of the projection}

As in Section 3.6.2 of \cite{GigPhi15}, we also need some suitable cut-off function and reparametrization function.

For any fixed $\hat{r},\hat{R}$ with $\hat{R}>r'>\hat{r}>0$,
pick a function $\hat{\varphi}\in C^{\infty}(\mathbb{R})$ with support in $(\frac{\hat{r}}{2},2\hat{R})$ so that
$$\hat{\varphi}(z)=(z-r')^{2} \qquad \text{ on }[\hat{r},\hat{R}].$$
Define $\hat{b}:X\rightarrow\mathbb{R}$ as $\hat{b}(x)=\hat{\varphi}\circ b(x)$.
Define the reparametrization function $\textmd{rep}:\mathbb{R}^{+}\times\mathbb{R}\rightarrow \mathbb{R}$, $\textmd{rep}_{t}(r)=\textmd{rep}(t,r)$ satisfying $\textmd{rep}_{0}(r)=0$ and
\begin{align}\label{rep_2}
\partial_{t}\textmd{rep}_{t}(r)=\hat{\varphi}'(r-\textmd{rep}_{t}(r))
\end{align}
for any $r$.
We then define the flow $\hat{F}:\mathbb{R^{+}}\times X\rightarrow X$ to be
$$
\hat{F}_{t}(x)=F_{-\textmd{rep}_{t}(b(x))}(x).
$$

Obviously, $\hat{F}_{t}$ is the identity on $b^{-1}((-\infty,\frac{\hat{r}}{2}])\cup b^{-1}([2\hat{R},\infty))$, $\hat{F}_{t}$ sends $b^{-1}([\hat{r},\hat{R}])$ into itself for every $t\geq0$;
$\hat{F}_{t}: X\rightarrow X$ is invertible for every $t\geq0$, and for $t > 0$ we define $\hat{F}_{-t}=(\hat{F}_{t})^{-1}$.
For every $x\in X$, let $\eta^{(x)}:[0,\infty)\mapsto X$ be the curve given by $\eta^{(x)}_{t}:=\hat{F}_{t}(x)$.
Properties for $\hat{F}_{t}$ similar to b) d) e) in Section 3.6.2 of \cite{GigPhi15} can be easily established.
Furthermore, it is easy to obtain
$$\frac{d}{dt}\hat{b}(\eta^{(x)}_{t})=-(\hat{\varphi}'\circ b)^{2}(\eta^{(x)}_{t})=-4\hat{b}(\eta^{(x)}_{t})\qquad\text{ if }x\in b^{-1}([\hat{r},\hat{R}]),$$
thus $|b(\eta^{(x)}_{t})-r'|$ decreases exponentially as $t\rightarrow+\infty$, hence $|\hat{F}_{t}(x)-\textmd{Pj}(x)|\rightarrow0$ uniformly for $x\in b^{-1}([\hat{r},\hat{R}])$ as $t\rightarrow+\infty$.
Similar to f) in Section 3.6.2 of \cite{GigPhi15}, we can derive that there exist positive constants $C$ and $\delta$ such that $\textmd{Lip}(\hat{F}_{s})\leq 1+|s|C$ for every $s\in[-\delta,\delta]$.

On the other hand, from the chain rule for Laplacian we are easy to see that $\hat{b}\in \textmd{Test}(X)$ and $\Delta\hat{b}\in L^{\infty}(X)$.
Furthermore, recall that by Lemma \ref{lem5.10}, $\textmd{Hess}(b)=0$ on $(b^{-1}(0,\infty))$, hence by the chain rules for Hessian, we have
\begin{align}\label{5.44}
\textmd{Hess}(\hat{b})=\hat{\varphi}'\circ b\textmd{Hess}(b)+\hat{\varphi}''\circ b\nabla b\otimes\nabla b=2\nabla b\otimes\nabla b,
\end{align}
holds on $b^{-1}([\hat{r},\hat{R}])$.

By the very similar arguments in Proposition 3.30 of \cite{GigPhi15}, using the properties of $\hat{b}$ and $\hat{F}_{t}$ here, we can establish the following proposition:

\begin{prop}\label{prop5.23}
Let $v\in L^{2}(TX)$ and put $v_{s}:=d\hat{F}_{s}(v)$.
Then the map $s\mapsto \frac{1}{2}|v_{s}|^{2}\circ\hat{F}_{s}\in L^{1}(X)$ is $C^{1}$ and its derivative is given by
\begin{align}\label{5.46}
\frac{1}{2}\frac{d}{ds}|v_{s}|^{2}\circ \hat{F}_{s}=-\textmd{Hess}(\hat{b})(v_{s},v_{s})\circ \hat{F}_{s},
\end{align}
the incremental ratios being converging both in $L^{1}(X)$ and $m$-a.e..
\end{prop}

We have the following corollary of Proposition \ref{prop5.23}:

\begin{cor}\label{cor5.24}
Let $v\in L^{2}(TX)$ be concentrated on $b^{-1}([\hat{r},\hat{R}])$ and put $v_{s}:=d \hat{F}_{s}(v)$. Then for every $s_{2}>s_{1}\geq0$ we have
\begin{align}\label{5.57}
|v_{s_{2}}|^{2}\circ \hat{F}_{s_{2}}\leq|v_{s_{1}}|^{2}\circ \hat{F}_{s_{1}},\qquad m\text{-a.e..}
\end{align}
\end{cor}

\begin{proof}
By Proposition \ref{prop5.23}, the map $s\mapsto \frac{1}{2}|v_{s}|^{2}\circ\hat{F_{s}}\in L^{1}(X)$ is $C^{1}$, and its derivative is given by
$\frac{d}{ds}|v_{s}|^{2}\circ \hat{F}_{s}=-2\textmd{Hess}(\hat{b})(v_{s},v_{s})\circ \hat{F}_{s}.$
Note that for $s\geq0$, $\frac{1}{2}|v_{s}|^{2}\circ\hat{F_{s}}$ is $0$ $m$-a.e. on $X\setminus b^{-1}([\hat{r},\hat{R}])$.
Furthermore by (\ref{5.44}),
$\textmd{Hess}(\hat{b})(v_{s},v_{s})\circ \hat{F}_{s}=2\langle\nabla b,v\rangle^{2}$ on $b^{-1}([\hat{r},\hat{R}])$.
Hence $\frac{d}{ds}|v_{s}|^{2}\circ \hat{F}_{s}\leq0$ holds for $s\geq0$ and we conclude.
\end{proof}

The following proposition can be proved by repeating verbatim the proof of Proposition 3.32 in \cite{GigPhi15}, Corollary \ref{cor5.24} is used in place of Corollary 3.31 in \cite{GigPhi15} in the argument.

\begin{prop}\label{prop5.25}
Let $[c,d]\subset(0,\infty)$ and $\pi$ be a test plan on $X$ such that $b(\eta_{t})\in[c, d]$ for every $t\in[0, 1]$ and $\pi$-a.e. $\eta$.
Then for $\pi$-a.e. $\eta$ the curve $\tilde{\eta}:= \textmd{Pj}\circ\eta$ is absolutely continuous and satisfies
\begin{align}\label{5.60}
|\dot{\tilde{\eta}}_{t}|\leq|\dot{\eta}_{t}|,\qquad  \text{for a.e. } t\in[0,1].
\end{align}
\end{prop}

\subsection{Sobolev Spaces on Y and X}

By Proposition \ref{prop5.25}, we can understand the Sobolev Spaces on $Z$ well.

\begin{prop}\label{prop5.26}
Let $[c,d]\subset(0,\infty)$, $h\in \textmd{Lip}(\mathbb{R})$ with $\textmd{supp}(h)\subset\subset(0,\infty)$ and identically $1$ on $[c,d]$.
Suppose $f\in L^{2}(X)$ is of the form $f(x)=g(\textmd{Pj}(x))h(b(x))$ for some $g\in L^{2}(Z)$.
Then $f\in W^{1,2}(X)$ if and only if $g\in W^{1,2}(Z)$ and in this case we have
\begin{align}\label{5.69}
|Df|_{X}(x)= |Dg|_{Z}(\textmd{Pj}(x))
\end{align}
for $m$-a.e. $x$ such that $b(x)\in [c,d]$.
\end{prop}

The proof can be carried out by an easy adaption from Theorem 3.24 of \cite{GigPhi15},
$|Df|_{X}(x)\geq|Dg|_{Z}(\textmd{Pj}(x))$ can be obtained directly from definitions as well as Theorem \ref{thm5.16}, while the proof of $|Df|_{X}(x)\leq |Dg|_{Z}(\textmd{Pj}(x))$ need to use Proposition \ref{prop5.25} in place of Proposition 3.32 in \cite{GigPhi15}.

In \cite{GigHan15}, the authors introduce the notion of `measured-length space', and prove that if a metric measure space is locally doubling and measured-length, then it has the Sobolev-to-Lipschitz property (see Proposition 3.18 in \cite{GigHan15}).

\begin{prop}
$(Z,d',m')$ is infinitesimally Hilbertian and a measured-length space.
\end{prop}
The infinitesimal Hilbertianity of $(Z,d',m')$ is a direct consequence of Proposition \ref{prop5.26}.
To check that $(Z,d',m')$ is a measured-length space, we just follow the lines in the proof of Proposition 3.26 in \cite{GigPhi15}, whose idea is to use the good properties of optimal transport on $X$ (see \cite{RaSt14} \cite{GRS13}), and use $\textmd{Pj}$ to map test plans on $X$ to test plans on $Z$.
Besides the good properties of optimal transport on $X$, the fundamental properties in Lemma \ref{lem5.1}, Proposition \ref{prop5.25} and that $\textmd{Pj}$ is locally Lipschitz are also used in the proof, we omit the details here.

Now we can apply the consequences of Section 3 in \cite{GigHan15} to conclude that $Y$ is infinitesimally Hilbertian and has the Sobolev-to-Lipschitz property.

Now we can compare the Sobolev spaces of $X$ and $Y$.
Given an open set $U\subset X$, we denote by $W^{1,2}_{0}(U)$ the $W^{1,2}(X)$-completion of the space of functions in $W^{1,2}(X)$ with support in $U$.
Similar notation is used for $W^{1,2}_{0}(V)$ with open set $V\subset Y$.

\begin{thm}\label{thm5.29}
Suppose $0<c<d$, then $f\in W^{1,2}_{0}(Y_{(c,d)})$ if and only if $f\circ S\in W^{1,2}_{0}(b^{-1}((c,d)))$ and in this case
\begin{align}\label{5.78}
|Df|_{Y}\circ S = |D(f\circ S)|_{X}, \quad m\text{-a.e. on } b^{-1}((c,d)).
\end{align}
\end{thm}

Theorem \ref{thm5.29} can be proved by the very same arguments in \cite{Gig13} or \cite{GigPhi15}.

\subsection{Back to the metric properties and conclusion}

\begin{proof}[Proof of Theorem \ref{thm1.2}]
We will handle Case (2) here.
We know $T$ and $S$ are measure preserving in (\ref{5.43}).
Combining Theorem \ref{thm5.29} with the fact that both $X$ and $Y$ have the Sobolev-to-Lipschitz property, the same argument as in the proof of Theorem \ref{thm5.16} can be applied to show that $T$ and $S$ are locally isometries.
We omit the details here.
Now we will derive the additional conclusions.

For any $x_{0},x_{1}\in b^{-1}((0,\infty))$, note that $(Y, d_{Y})$ is a geodesic space, thus there is always a shortest geodesic $\sigma:[0,1]\rightarrow Y$ with $\sigma(0)=S(x_{0})$, $\sigma(1)=S(x_{1})$.
Since $T$ is a local isometry, $T\circ\sigma:[0,1]\rightarrow X$ is a Lipschitz curve of length $d_{Y}(S(x_{0}),S(x_{1}))$ connecting $x_{0}$ and $x_{1}$, thus $d(x_{0},x_{1})\leq d_{Y}(S(x_{0}),S(x_{1}))$.

Now let $r_{1}=\frac{\textmd{diam}'(Z)}{2}$, and suppose $x_{0},x_{1}\in b^{-1}((r_{1},\infty))$.
Let $\eta:[0,1]\rightarrow X$ be a shortest geodesic with $\eta(0)=x_{0}$, $\eta(1)=x_{1}$.
$\eta$ has to be contained in $b^{-1}((0,\infty))$, for otherwise there is a point $z=\eta_{t}\in b^{-1}((-\infty,0])$, then
\begin{align}
&d(x_{0},x_{1})=d(z,x_{0})+d(z,x_{1})\geq b(x_{0})+b(x_{1})\nonumber\\
> &|b(x_{1})-b(x_{0})|+\textmd{diam}'(Z)\geq d_{Y}(S(x_{0}),S(x_{1})),\nonumber
\end{align}
which is a contradiction.
Now $S\circ\eta:[0,1]\rightarrow Y$ is a Lipschitz curve of length $d(x_{0},x_{1})$, and hence $d(x_{0},x_{1})\geq d_{Y}(S(x_{0}),S(x_{1}))$.
Thus $S:(b^{-1}((r_{1},\infty)),d)\rightarrow(Z\times(r_{1},\infty), d_{Y})$ is an isometry.

Note that for any $x_{0},x_{1}\in b^{-1}((r_{1},\infty))$, suppose $\sigma:[0,1]\rightarrow Z\times(r_{1},\infty)$  is a shortest geodesic connecting $S(x_{0})$ and $S(x_{1})$, then $T\circ \sigma$ is a shortest geodesic connecting $x_{0}$ and $x_{1}$, hence $b^{-1}((r_{1},\infty))$ is a geodesic space.
Thus by Proposition 7.7 in \cite{AMS14}, $(b^{-1}((r_{1},\infty)),d,m)$ is $\textmd{RCD}(0,N)$.
By the isomorphism $S:(b^{-1}((r_{1},\infty)),d,m)\rightarrow(Z\times(r_{1},\infty), d_{Y}, m_{Y})$ and a natural isomorphism between $(Z\times(r_{1},\infty), d_{Y}, m_{Y})$ and $(Y, d_{Y}, m_{Y})$, we know $(Y, d_{Y}, m_{Y})$ is also an $\textmd{RCD}(0,N)$ space.
By Corollary 2.5 in \cite{St06II}, the Hausdorff dimension of $(Y, d_{Y}, m_{Y})$ is at most $N$.
On the other hand, from Corollary \ref{cor5.17}, $Y$ has Hausdorff dimension at least $2$, hence $N\geq2$.

Since $(Y, d_{Y}, m_{Y})$ is an $\textmd{RCD}(0,N)$ space, and $(Y, d_{Y}, m_{Y})$ is the product of $(Z, d', m')$ and $(\mathbb{R}^{+},d_{Eucl},\mathcal{L}^{1})$, the argument in Corollary 5.30 and Theorem 7.4 of \cite{Gig13} shows that $(Z,d',m')$ is an $\textmd{RCD}(0,N-1)$ space.

The proof is completed.
\end{proof}

\end{document}